\documentclass[11pt]{article}
\usepackage{hyperref, fullpage, amsthm, amsmath, wrapfig, bm, amssymb}
\usepackage{caption}
\usepackage{subcaption}
\usepackage{tikz}
\usetikzlibrary{decorations.pathmorphing}
\usetikzlibrary{shapes.geometric, decorations.markings}

\tikzset{mysquare/.style={regular polygon,regular polygon sides=4, inner sep=0}} 
\tikzset{mypent/.style={regular polygon,regular polygon sides=5, inner sep=0}} 
\tikzset{myhex/.style={regular polygon,regular polygon sides=6, inner sep=0}} 
\tikzstyle{5vert}=[mysquare, minimum size = 4.5pt, inner sep = 0.5pt,
outer sep = 0pt, draw, fill=white]
\tikzstyle{6vert}=[shape = circle, minimum size = 8.0pt, inner sep = 2pt,
outer sep = 0pt, draw, fill=white]
\tikzstyle{7vert}=[mypent, minimum size = 8.0pt, inner sep = 2pt,
outer sep = 0pt, draw, fill=white]
\tikzstyle{unvert}=[draw=none, fill=white]
\tikzstyle{uvert}=[draw=none, fill=none]
\tikzset{every node/.style=6vert}

\tikzstyle{5vertB}=[mysquare, minimum size = 15.0pt, inner sep = 1.5pt,
outer sep = 0pt, draw, fill=white]
\tikzstyle{6vertB}=[shape = circle, minimum size = 15.0pt, inner sep = 1.5pt,
outer sep = 0pt, draw, fill=white]
\tikzstyle{7vertB}=[mypent, minimum size = 15.0pt, inner sep = 0.5pt,
outer sep = 0pt, draw, fill=white]
\tikzstyle{8vertB}=[myhex, minimum size = 15.0pt, inner sep = 0.5pt,
outer sep = 0pt, draw, fill=white]
\tikzstyle{unvertB}=[draw=none, fill=white]
\tikzstyle{uvertB}=[draw=none, fill=none]

\def\a{\alpha}
\def\b{\beta}

\def\ceil#1{\lceil#1\rceil}
\def\Ceil#1{\left\lceil#1\right\rceil}

\def\Jc{{\mathcal J}}
\def\diam{\textrm{diam}}
\def\prune{\textrm{prune}}

\def\CL#1{\mathcal{C}_L(#1)}
\def\dvorak{Dvo\v{r}\'{a}k}

\newtheorem{lem}{Lemma}

\newtheorem{prop}[lem]{Proposition}
\newtheorem{conj}[lem]{Conjecture}
\newtheorem{obs}[lem]{Observation}
\newtheorem*{mainthm}{Main Theorem}
\theoremstyle{definition}
\newtheorem{rem}[lem]{Remark}

\RequirePackage{marginnote,hyperref}
\newcommand{\aside}[1]{\marginnote{\scriptsize{#1}}[0cm]}
\newcommand{\aaside}[2]{\marginnote{\scriptsize{#1}}[#2]}
\newcommand\Emph[1]{\emph{#1}\aside{#1}}
\newcommand\EmphE[2]{\emph{#1}\aaside{#1}{#2}}

\makeatletter
 \newcommand{\linkdest}[1]{\Hy@raisedlink{\hypertarget{#1}{}}}
\makeatother
\newcommand\Rule[1]{\hyperlink{R#1}{(R#1)}}

\def\vph{\varphi}
\renewcommand\le\leqslant
\renewcommand\ge\geqslant

\title{$10$-list Recoloring of Planar Graphs}

\author{Daniel W. Cranston\thanks{%
Dept.~of Computer Science, Virginia Commonwealth
University, Richmond, VA, USA;
\texttt{dcransto@gmail.com}
} 
}
\begin{document}
\maketitle

\begin{abstract}
Fix a planar graph $G$ and a list assignment $L$ with $|L(v)|=10$ for all $v\in V(G)$.
Let $\a$ and $\b$ be $L$-colorings of $G$.  A recoloring sequence from $\a$ to $\b$
is a sequence of $L$-colorings, beginning with $\a$ and ending with $\b$, such that each successive
pair in the sequence differs in the color on a single vertex of $G$.  We show that there exists a constant $C$
such that for all choices of $\a$ and $\b$ there exists a recoloring sequence $\sigma$ from $\a$ to $\b$
that recolors each vertex at most $C$ times.  In particular, $\sigma$ has length at most $C|V(G)|$.
This confirms a conjecture of {\dvorak} and Feghali.
For our proof, we introduce a new technique for quickly showing that many configurations are reducible.  We believe
this method may be of independent interest and will have application to other problems in this area.
\end{abstract}

\section{Introduction}
A recoloring sequence transforms one specified $k$-coloring of a graph $G$ into another by recoloring one vertex at
a time.  We also require that each recoloring step yields an intermediate proper $k$-coloring of $G$.  We typically ask whether
every $k$-coloring can be transformed (or ``reconfigured'') to every other.  If so, we seek to bound as a function of $|G|$ the maximum
number of recoloring steps needed, over all pairs of $k$-colorings.

Analogues of this paradigm have been applied to a broad range of structures including independent sets, dominating sets, solutions
to 3-SAT, triangulations of a planar point set, and non-crossing spanning trees, to name a few.  This general topic of study is called 
\emph{reconfiguration}, and we recommend to the interested reader the surveys~\cite{nishimura-survey} and~\cite{vdH-survey} and the introduction of~\cite{BKUV}.
Perhaps the most widely studied area within reconfiguration is graph coloring, and that is our focus here.

Cereceda suggested~\cite[Conjecture~5.21]{cereceda-diss} that if a graph $G$ is $d$-degenerate and $k\ge d+2$, then each 
$k$-coloring $\a$ of $G$ can be reconfigured to each other $k$-coloring $\b$ of $G$ with a recoloring sequence of length $O(|G|^2)$.  
This problem has attracted much interest and work, including some impressive partial results~\cite{BH-poly-cereceda's,feghali-JCTB}.  
Nonetheless, the problem remains wide open, even in the case $d=2$.  

In this paper, we consider a list-coloring variant of this conjecture, and we make the lists large enough to guarantee a 
recoloring sequence with length $O(|G|)$.
Specifically, we focus on graph classes $\mathcal{G}$ and sizes $k$ such that there exists a constant $C_{k,\mathcal{G}}$ satisfying
the following: For every graph $G\in \mathcal{G}$ and every $k$-assignment $L$ and every two $L$-colorings $\alpha$ and $\beta$, there
exists a recoloring sequence from $\a$ to $\b$ that recolors each vertex at most $C_{k,\mathcal{G}}$ times.
Typically, we take $\mathcal{G}$ to be a hereditary class. To prove such a length bound of $O(|G|)$ by induction, we find a 
``good subgraph'' $H$ in $G$, claim the result for the smaller graph $G-H$, by the induction hypothesis, and then show 
that we can extend the result to $G$ precisely because we chose $H$ to be ``good''.  For these types of arguments, the 
following Extension Lemma is invaluable.

\begin{lem}[Extension Lemma]
\label{extension:lem}
Let $G$ be a graph, $L$ be a list assignment for $G$, and $\a$ and $\b$ be 
$L$-colorings for $G$.  Fix $v\in V(G)$.  Let $G':=G-v$, and let $\a'$ and $\b'$
be the restrictions to $G'$ of $\a$ and $\b$.  Suppose there exists a recoloring
sequence $\sigma'$ that recolors $G$ from $\a'$ to $\b'$.  If $\sigma'$ recolors
vertices in $N(v)$ a total of $s$ times, then $\sigma'$ extends to a recoloring
sequence $\sigma$ for $G$ from $\a$ to $\b$ that recolors $v$ at most
$\Ceil{\frac{s}{|L(v)|-d(v)-1}}+1$ times.
\end{lem}

\begin{proof}
Let $c_1,\ldots,c_s$ denote the colors that are used by $\sigma'$ to recolor
$N(v)$, with multiplicity and in order.  We show how to extend $\sigma'$ to
$\sigma$.  Each time that $\sigma$ will recolor $N(v)$ with some $c_i$ that is
currently used on $v$, we must first recolor $v$
before we can continue with the recoloring step in $\sigma'$.  Let
$a:=|L(v)|-d(v)-1$.  We must avoid all colors currently used on
$N(v)$, as well as $c_i$, so we have at least $a$ choices.  We choose a color
that does not appear among $c_i,c_{i+1},\ldots,c_{i+a-1}$.  Thus, the number of
times that we must recolor $v$, due to its current color being needed on
$N(v)$, is at most $\Ceil{\frac{s}{|L(v)|-d(v)-1}}$.  To conclude, we may need
one more step to recolor $v$ to match $\b(v)$.
\end{proof}

The Extension Lemma has been used implicitly in many papers; it 
was first stated explicitly for coloring in~\cite{BBFHMP}, and generalized to
list-coloring (with the same proof) in~\cite{recoloring-sparse}.
We note also that the same proof works in the more general context of
``correspondence'' coloring; however, we omit the definitions needed to
make this formal, since including them would only add more abstraction without
adding any more interesting ideas.  (In fact, nearly all of the arguments in 
this paper work equally well for correspondence coloring.  However, we will
not mention this in what follows.)

For a graph $G$ and list assignment $L$, let \Emph{$\CL{G}$} denote the graph with 
a node for each $L$-coloring of $G$ and an edge between each pair of nodes that 
differ only in the color of a single vertex of $G$.
Bousquet and Perarnau~\cite{BP} used the Extension Lemma to prove the following
pretty result.

\begin{lem}[\cite{BP}] 
\label{linear-cereceda:lem}
If $G$ is a $d$-degenerate graph and $L$ is a list assignment with
$|L(v)|\ge 2d+2$ for all $v\in V(G)$, then
$\diam~\CL{G}\le (d+1)|G|$.
\end{lem}
\begin{proof}
Let $L$ be a list assignment as in the lemma, and fix $L$-colorings $\a$ and $\b$.
We prove that $G$ can be recolored from $\a$ to $\b$ such that each vertex is
recolored at most $d+1$ times.  Pick $v$ such that $d(v)\le d$.  Let $G':=G-v$,
and let $\a'$ and $\b'$ denote the restrictions to $G'$ of $\a$ and $\b$.
By hypothesis, there exists a sequence $\sigma'$ that recolors $G$ from $\a'$ to
$\b'$ and recolors each vertex at most $d+1$ times.  By the Extension Lemma, we
can extend $\sigma'$ to a sequence $\sigma$ that recolors $v$ at most
$\Ceil{\frac{|N(v)|(d+1)}{2d+2-(d+1)}}+1=|N(v)|+1\le d+1$ times.
\end{proof}

\begin{conj}
Let $k,d$ be positive integers with $k\ge d+3$.  There exists a constant\footnote{A priori, we might want to allow $C_d$ to also 
depend on $k$, but the stated version is equivalent, as follows.  If the weaker version is true, then $C^k_d$ exists for each 
$k\ge d+3$.  Let $C_d:=\max\{d+1,\max_{d+3\le k\le 2d+1}C^k_d\}$. This suffices by the previous lemma.} $C_d$ such that if $G$ is 
a $d$-degenerate graph and $L$ a $k$-assignment for $G$, then $\diam~\CL{G}\le C_d|G|$.
\end{conj}

The conjecture above was posed in~\cite{BBFHMP} only for coloring (the case when $L(v)=\{1,\ldots,k\}$ for all 
$v$), but here we extend it to list-coloring.  In fact, we propose the following meta-conjecture.

\begin{conj}
\label{meta-conj}
For every ``natural'' graph class $\mathcal{G}$ and positive integer $k$, if there exists $C_{k,\mathcal{G}}$ such that 
$\diam~\CL{G} 
\le C_{k,\mathcal{G}}|G|$ for all $G\in \mathcal{G}$ when $L(v)=\{1,\ldots,k\}$ for all $v$, then also there exists $C'_{k,\mathcal{G}}$ such that $\diam~\CL{G}\le C'_{k,\mathcal{G}}|G|$ for all $G\in \mathcal{G}$ for every $k$-assignment $L$.
\end{conj}

In particular, we pose this conjecture when $\mathcal{G}$ is (i) the class of all graphs with maximum average 
degree less than $a$ and (ii) the class of all planar graphs with girth at least $g$.
As evidence in support of Conjecture~\ref{meta-conj} holding in case (ii), we summarize below how many colors are needed
to imply that $\diam~\CL{G} = O(|G|)$ for planar graphs with various lower bounds on girth.

\begin{table}[!h]
\centering
\begin{tabular}{c|ccccc}
girth $\ge$ & 3 & 4 & 5 & 6 & 11\\
\hline
\# colors & 10 & 7 & 6 & 5 & 4 
\end{tabular}
\captionsetup{justification=centering}
\caption{Best known bounds on list sizes needed for planar\\ graphs of various girths to ensure $\diam~\CL{G}=O(|G|)$.~~~~}
    \end{table}

In the case of coloring (without lists), the first two columns were proved by {\dvorak} and Feghali~\cite{DF2,DF1}.
They also conjectured the third column.  Furthermore, they conjectured that the first two columns also hold
in the more general context of list coloring.  

In~\cite{BBFHMP}, its authors proved the fourth column.  And their proof also works for list-coloring.
In the same paper, they revisited the analogue of the first column for list-coloring.  
Since planar graphs are 5-degenerate, Lemma~\ref{linear-cereceda:lem} implies the desired bound for lists of size 12.  
The authors extended this result to lists of size 11 using the Extension Lemma and the following result of Borodin: 
If $G$ is planar with minimum degree 5, then $G$ has a 3-face $vwx$ with $d(v)+d(w)+d(x)\le 17$.

In~\cite{recoloring-sparse}, the present author proved the third and fifth columns, even for list-coloring.
The same paper proved the list-coloring analogue of the second column.  However, the problem of a list-coloring analogue
of the first column, as conjectured by {\dvorak} and Feghali~\cite{DF1}, remained open.
The authors of~\cite{planar-forbidden} proved the conjecture for certain planar graphs with 
maximum average degree less than $5.\overline{3}$.
Our main contribution in this paper confirms this conjecture for all planar graphs.

\begin{mainthm}
There exists a constant $C$ such that if $G$ is a planar graph with a $10$-assignment $L$ and with $L$-colorings $\a$ and $\b$,
then there exists a recoloring sequence from $\a$ to $\b$ that recolors each vertex at most $C$ times.
\end{mainthm}

A \Emph{configuration} is a subgraph $H$ along with prescribed values of $d_G(v)$ for each vertex $v$ in $H$.
A configuration is \Emph{reducible} if it cannot appear in a $k$-minimal counterexample $G$ (see the start
of the Section~\ref{reducibility-sec}).
Our proof of the Main Theorem uses the discharging method, with 36 reducible configurations; 
Figures~\ref{node-shapes-fig},\ref{main-reduc-fig},\ref{more-reduc-fig} show example configurations.  At first glance, 
verifying reducibility for so many configurations seems to be a daunting task.  However, we introduce a new technique to handle 
most of the configurations in just over 4 pages, without any reliance on computer verification.  

A vertex $w_2$ \emph{defers to} a neighbor $w_1$ if $w_1$, when it must be recolored, need not avoid the color $c$ on $w_2$.  If
$w_1$ wants to use $c$, then it can, but $w_2$ must first be recolored to avoid $c$.  When we apply the Extension Lemma to a vertex
$v$, that vertex implicitly defers to each of its neighbors in $G-v$.  Our technique allows for the possibility that, with neighbors 
$w_1$ and $w_2$, neither defers to the other; rather, whenever we recolor $w_1$ we must avoid the color currently on $w_2$, and 
vice versa.
We also introduce a method to quickly bound the total number of times that a vertex will be recolored in this more general context.
Specifically, we are able to bound the number of recoloring steps needed for various vertices in a configuration; and these 
computations largely decouple to be independent of each other.  As a result, we can rapidly generate many configurations that
cannot appear in our minimal counterexample.

We believe this technique will be of independent interest, and will likely
have application to other problems in this area.  In Section~\ref{reducibility-sec} we show that various configurations are 
reducible, and in Section~\ref{unavoidability-sec} we use discharging to show that every planar graph contains at
least one of these reducible configurations, and hence cannot be a $k$-minimal counterexample.

Our terminology and notation are standard, however we highlight a few terms.  If $d(v)=k$, for some vertex $v$, then $v$
is a \Emph{$k$-vertex}; if $d(v)\ge k$ or $d(v)\le k$, then $v$ is respectively a \emph{$k^+$-vertex} or a 
\emph{$k^-$-vertex}.\aside{$k^+/k^-$-vertex}
If $v$ is a $k$-vertex and $v$ is a neighbor of some vertex $w$, then $v$ is a \EmphE{$k$-neighbor}{5mm}; analogously, we define 
\emph{$k^+$-neighbor} and \emph{$k^-$-neighbor}.  A \emph{planar graph} is one that can be embedded in the plane with no edge
crossings.  A \Emph{plane graph} is a fixed plane embedding of a planar graph.
A graph is \Emph{$d$-degenerate} if every subgraph $H$ has some $v\in V(H)$ such that $d_H(v)\le d$.

\section{Reducibility}
\label{reducibility-sec}

Fix a positive integer $k$.  
A recoloring sequence is \Emph{$k$-good} if every vertex is recolored at most $k$ times.
We aim to prove the follow: Fix a plane graph $G$ and a $10$-assignment $L$ for $G$.
For any two $L$-colorings $\a$ and $\b$ of $G$, there exists a $k$-good recoloring sequence $\sigma$ from $\a$ 
to $\b$.  A \Emph{$k$-minimal counterexample} (to this statement) is a 
graph $G$ such that the statement is false, but the statement is true for every plane graph with fewer vertices than $G$.
We assume that $G$ is a $k$-minimal counterexample, and reach a contradiction.

A \Emph{reducible configuration} (for some fixed $k$) is a subgraph $H$, with prescribed values $d_G(v)$ for all $v\in V(H)$,
that cannot appear in a $k$-minimal counterexample.  In this section, we compile a large set of reducible configurations.
In the next section, we use a counting argument (via the discharging method) to show that every plane graph contains at least
one of those reducible configurations.  Thus, no plane graph $G$ is a $k$-minimal counterexample.  So the Main Theorem holds.

The discharging in
our proof would be simpler if we knew that $G$ was a triangulation.  For many coloring results on planar graphs, this assumption
is easily justified; if $G$ has a non-triangular face, then we add some diagonal of the face, and each coloring of the new graph
is a coloring of the original.  But our present problem is more subtle, since we are not searching for a single coloring, but rather
for a short path between every two colorings.  Suppose that $G$ has a face $v_1v_2v_3v_4$ and a coloring $\a$ with 
$\a(v_1)=\a(v_3)$ and $\a(v_2)=\a(v_4)$.  By adding either $v_1v_3$ or $v_2v_4$ we form a graph for which $\a$ is no longer a
proper coloring.  Nonetheless, we show (see Lemma~\ref{edge-remove-lem} and the comment at the start of 
Section~\ref{discharging-sec}) that we can triangulate $G$, only for the discharging.  If we find a reducible configuration $H$ 
in this triangulation $G^+$, then $H$ also restricts to a reducible configuration in $G$.

Recall that in the proof of the Extension Lemma, we might need a single step at the end to recolor the deleted vertex $v$ to 
agree with $\beta$.  In the present more general context, this final recoloring phase is handled by our next proposition.

\begin{prop}
Fix a plane graph $G$, a subgraph $H$, a 10-assignment $L$, and $L$-colorings $\a$ and $\b$ that differ only on $H$. 
(i) Now $G$ has a recoloring
sequence from $\a$ to $\b$ that recolors each vertex of $H$ at most twice (and recolors no other vertices) whenever there exists
an order $\sigma$ of $V(H)$ such that $d_G(v)+d_{\sigma}(v)\le 9$ for all $v\in V(H)$, where $d_{\sigma}(v)$ is the number of 
neighbors of $v$ that follow $v$ in $\sigma$.  (ii) In particular, this is true whenever $H$ is 2-degenerate and $d_G(v)\le 7$
for all $v\in V(H)$.
\label{finishing-prop}
\end{prop}
\begin{proof}
We use induction on $|H|$, and the base case, $|H|=1$, is trivial.  For the induction step, let $v$ be the first vertex in $\sigma$
and let $H':=H-v$.  Recolor $v$ to avoid all colors currently used on its neighbors, as well as all colors used on its neighbors in 
$\b$. (This is possible because $d_G(v)+d_H(v)=d_G(v)+d_{\sigma}(v)\le 9$, by hypothesis.)  By the induction hypothesis, recolor $H'$
to agree with $\b$, treating $v$ as a vertex of $G$.  Finally, recolor $v$ to agree with $\b$. (ii) Since $H$ is 2-degenerate,
there exists $\sigma$ such that $d_{\sigma}(v)\le 2$ for all $v\in V(H)$.  Thus, $d_G(v)+d_{\sigma}(v)\le 7+2=9$.  So we are done
by (i).
\end{proof}

\begin{rem}
Throughout this paper, in various subgraphs $H$ we use vertex shapes to denote vertex degrees in a larger graph $G$, as follows;
see Figure~\ref{node-shapes-fig}.
Degree 5 is a square, degree 6 is a circle, degree 7 is a pentagon, and degree 8 is a hexagon.  If only a lower bound on degree
is known, then the corresponding node is drawn as dashed.  For example, a vertex of degree at least 6 is drawn as a dashed
circle.  If no nontrivial degree bound is known (or our argument does not depend on such a bound), then we simply do not draw
a shape for the node.  This (absence of shape) will not occur in the present section on reducibility, but will occur 
frequently in the section on unavoidability.  For example, in Figure~\ref{5-86766fig}: vertex $v$ is a $5$-vertex, $w_4$ is a 
$7$-vertex, $w_1$ is an $8^+$-vertex; for each remaining vertex, either it is a $6$-vertex, or we claim no bound on its degree.
\end{rem}
\smallskip

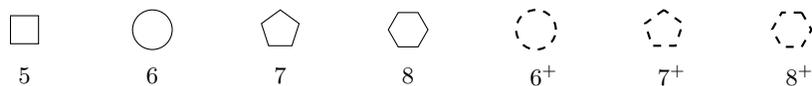
\begin{figure}[!ht]
\centering
\begin{tikzpicture}[yscale=.866, xscale=1.7]
\tikzset{every node/.style=6vertB}
\def\off{7mm}

\draw (0,0) node[5vertB] {};
\draw (0,-\off) node[draw=none] {\footnotesize{$5$}};
\draw (1,0) node[6vertB] {};
\draw (1,-\off) node[draw=none] {\footnotesize{$6$}};
\draw (2,0) node[7vertB] {};
\draw (2,-\off) node[draw=none] {\footnotesize{$7$}};
\draw (3,0) node[8vertB] {};
\draw (3,-\off) node[draw=none] {\footnotesize{$8$}};
\draw (4,0) node[6vertB,dashed, thick] {};
\draw (4,-\off) node[draw=none,fill=none] {\footnotesize{~~$6^+$}};
\draw (5,0) node[7vertB,dashed, thick] {};
\draw (5,-\off) node[draw=none,fill=none] {\footnotesize{~~$7^+$}};
\draw (6,0) node[8vertB,dashed, thick] {};
\draw (6,-\off) node[draw=none,fill=none] {\footnotesize{~~$8^+$}};
\end{tikzpicture}
\caption{Vertex degree as denoted by shape\label{node-shapes-fig}}
\end{figure}

Before introducing our general approach, we first illustrate it on one particular small configuration.  The general outline of the
proof draws heavily on ideas from the proof of the Extension Lemma.  However, that lemma is not well-suited for the case of
$k$-assignments when $k$ is even (such as 10).  And it is easy to check that we cannot use it directly to prove reducibility
of any of our reducible configurations.  So, rather than deleting a single vertex at a time, we instead delete all vertices in
the configuration $H$ at once.  When we extend our recoloring sequence from $G-H$ to $G$, we generally do not have any
of the vertices of $H$ ``defer'' to each other.  That is, for vertices $v,w\in H$ with $vw\in E(H)$, if vertex $v$ 
must be recolored due to some
upcoming recoloring in $N_G(v)\setminus N_H(v)$, then $v$ must avoid the color currently on $w$.  But also $w$ will need to avoid
the color used on $v$ when roles are reversed.  (Still, there are some exceptions to this general rule.)

\begin{lem}
The subgraph $H$ in Figure~\ref{first-reduc-fig} 
appears in no $k$-minimal counterexample with $k\ge 828$.
\label{first-reduc-lem}
\end{lem}

\noindent
\textit{Proof.}
Suppose the lemma is false.
Fix a plane graph $G$, an integer $k\ge 828$, a $10$-assignment $L$, and $L$-colorings $\a$ and $\b$ witnessing this.
Denote $V(H)$ by $v,w,x,y,z$, as shown in Figure~\ref{first-reduc-fig}.  
Let $G':=G-V(H)$.  Let $\a'$ and 
$\b'$ denote the restrictions to $G'$ of $\a$ and $\b$.  Since $G$ is $k$-minimal, there exists a 
$k$-good recoloring sequence $\sigma'$ from $\a'$ to $\b'$.  We show how to extend 
$\sigma'$ to a $k$-good recoloring sequence for $G$ from $\a$ to $\b$.
By Proposition~\ref{finishing-prop}, once we have recolored $G$ to match $\b$ on $V(G)\setminus V(H)$, we can recolor
$V(H)$ to match $\b$ by recoloring each vertex of $H$ at most 2 times.

Let $a_1,\ldots$ denote the sequence of colors that appear on $N_G(v)\setminus N_H(v)$, in order, during $\sigma'$.
Similarly, we define $b_1,\ldots$ for $w$; $c_1,\ldots$ for $x$; $d_1,\ldots$ for $y$; and $e_1,\ldots$ for $z$.  
To extend $\sigma'$ to $\sigma$, each time that the current color on $y$, 
call it $\vph(y)$ appears as some $d_i$ on some vertex in $N_G(y)\setminus N_H(y)$, we first recolor vertex 
$y$, and possibly some other vertices in $H$ to avoid $d_i$.  When we are recoloring $y$ to avoid $d_i$, we must avoid the
colors currently used on all neighbors of $y$ (both outside and inside $H$).  However, since $|L(y)|\ge 10$ and $d(y)=5$,
we have at least $10-5=5$ available colors for $y$, and we can choose one to avoid $\{d_i, d_{i+1}, d_{i+2}, d_{i+3}\}$.
By hypothesis, each neighbor of $y$ is recolored at most $k$ times in $\sigma'$, so when we extend $\sigma'$
to $H$ vertex $y$ is recolored at most $4k/4 = k$ times.  However, we may need $2$ more recoloring steps at the end
of this sequence to get the correct colors on $H$.  So we need something slightly more subtle.

\begin{wrapfigure}{l}{0.3\textwidth}
\centering
\begin{tikzpicture}[yscale=.866, xscale=.5, scale=1.0]
\tikzstyle{5vert}=[mysquare, inner sep = 2pt, outer sep = 0pt, draw, fill=white]
\tikzstyle{6vert}=[shape = circle, minimum size = 15.0pt, inner sep = 2pt, outer sep = 0pt, draw, fill=white]
\tikzstyle{7vert}=[mypent, minimum size = 15.0pt, inner sep = 2pt, outer sep = 0pt, draw, fill=white]
\tikzstyle{unvert}=[draw=none, fill=white]
\tikzstyle{uvert}=[draw=none, fill=none]
\tikzset{every node/.style=6vert}
\draw (0,0) node[5vert] (a) {$v$} -- (1,1) node (b) {$w$} -- (2,0) node (c) {$z$} -- (3,1) node (d) {$x$} -- (4,0) node[5vert] (e) {$y$} (b) -- (d) (a) -- (c);
\end{tikzpicture}
\captionsetup{width=.22\textwidth}
\caption{This configuration appears in no $k$-minimal counterexample with $k\ge 828$.\label{first-reduc-fig}}
\end{wrapfigure}
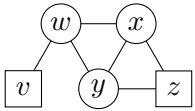

While $i\le 40$, when we need to recolor $y$ to avoid color $d_i$, we do not require $y$ to avoid the color currently
on $x$; thus $y$ has at least $5$ available colors, and can avoid the colors on $N_G(y)\setminus\{x\}$, as well as  
$d_i, d_{i+1}, d_{i+2}, d_{i+3}, d_{i+4}$.  If we want to recolor $y$ with the color $\vph(x)$ on $x$, then we 
first recolor $x$ to avoid $\vph(x)$ and all colors currently used on $N(x)$, as well as $c_{j}$ and $c_{j+1}$, the next
two colors appearing on $N(x)\setminus N_H(x)$.  After recoloring $x$, we recolor $y$ with $\vph(x)$ as desired.
So the number of times we recolor $y$ is at most $(4k-40)/4 + 40/5+2 = k$.  (The `+2' comes from Proposition~\ref{finishing-prop}.)

Now we do something similar for $x$.  In general, when we need to recolor $x$ to avoid $c_i$, we do it to avoid each color
currently on $N(x)$, as well as to avoid $c_i, c_{i+1}, c_{i+2}$.  However, when $i\le 144$ we do not require that $x$ avoids
the color $\vph(w)$ currently on $w$.  As above, if we want to recolor $x$ with $\vph(w)$, then we first recolor $w$ to avoid
each color currently on $N[w]$, as well as $b_{j}$ and $b_{j+1}$.  After recoloring $w$, we recolor $x$ with $\vph(w)$.
Thus, the number of times we recolor $x$ is at most $(3k-144)/3 + (144+40/5)/4+40/5+2 = k$.
Here the first term comes from recoloring $x$ to avoid colors $b_{145},\ldots$; 
the third from recoloring if $y$ wants the color on $x$; 
and the fourth from recoloring $x$ at most twice at the end to get the desired colors on $V(H)$.
Below we consider the second term.

The $144$ in the numerator is clear, but the `$+40/5$' is more subtle.  
The issue is that when $x$ defers to $y$, vertex $x$ can only avoid $c_i, c_{i+1}, c_{i+2}$, the next
3 colors that will appear on $N(w)\setminus N_H(w)$.  But if $x$ was just recolored to avoid the next $4$ colors (since $i\le 144$), then $x$ will
effectively ``lose ground'' of one color in this sequence;  specifically, this happens if $x$ must defer to $y$ before color $c_i$
appears in $N_G(x)\setminus N_H(x)$.
This loss of ground occurs at most $40/5=8$ times, 
once each time that $x$ must be recolored because it defers to $y$.
(So $w$ defers to $x$ at most $(144+8)/4=152/4$ times.)

For $w$, in general when we need to recolor $w$ to avoid $b_i$, we do it to avoid each color currently on $N(w)$, as well
as to avoid $b_i, b_{i+1}, b_{i+2}$.  However, when $i\le 594$ we do not require that $w$ avoids the color $\vph(v)$ currently
on $v$; and we recolor $v$ first if we want to recolor $w$ with $\vph(v)$.
Thus, the number of times we recolor $w$ is at most $(3k-594)/3 + (594+(152/4))/4+152/4+2 = k$.

For $z$, in general when we need to recolor $z$ to avoid $e_i$, we do it to avoid each color currently on $N(z)$, as well
as to avoid $e_i, e_{i+1}, e_{i+2}$.  However, when $i\le 24$ we do not require that $z$ avoids the color $\vph(v)$ currently
on $v$; and we recolor $v$ first if we want to recolor $z$ with $\vph(v)$.
Thus, the number of times we recolor $z$ is at most $(3k-24)/3 + 24/4+2 = k$.

For $v$, whenever we need to recolor $v$ to avoid $a_i$, we do it to avoid each color currently on $N(v)$, as well
as to avoid $a_i, a_{i+1}, a_{i+2}, a_{i+3}$.  
Thus, the number of times we recolor $v$ is at most $\ceil{(3k+(632/4)+(24/4))/4} + 632/4 + 24/4 + 2\le k$; 
here we use that $k\ge 828$.

We remark that it is possible that some vertex $u$ outside $H$ is about to be recolored with a color currently used on two
vertices of $H$.  In that case, we can break the tie arbitrarily.  For example if the color is currently used on vertices $v$
and $y$, then we first remove it from $v$ and later remove it from $y$, or vice versa.
(The proof also works if vertices $v$ and $y$ are identified, but we defer these details to the end of the proof of 
Lemma~\ref{main-reduc-lem}.)
\hfill$\qed$

\bigskip

Lemma~\ref{main-reduc-lem} is the main reducibility lemma proved in this section.  But before we get to it,
we address one other case, which is not quite handled by the general arguments used to prove Lemma~\ref{main-reduc-lem}.

\begin{lem} 
A $7^-$-vertex with no $7^+$-neighbor is $k$-reducible (for all $k\ge 44$).
\label{second-reduc-lem}
\end{lem}
\begin{proof}
Suppose our $k$-minimal counterexample $G$ contains a $7^-$-vertex $v$ with only $6^-$-neighbors.  
Denote these neighbors by $w_1,\ldots,w_{d(v)}$.  As in the previous lemma, we get a recoloring sequence $\sigma'$
for $G-N[v]$ and extend it to $N[v]$.  For each of the first 24 colors that appears in $N_G(w_1)\setminus N_{N[v]}(w_1)$
if we want to recolor $w_1$ with the color currently on $v$, then we do so, after first recoloring $v$ to avoid all
colors currently used on $N[v]$.  Thus, the number of times that $w_1$ is recolored is at most $(3k-24)/3+24/4+2=k$.
We handle each other $w_i$ identically.  Thus, the number of times that $v$ is recolored is at most $(24/4)d(v)+2\le 44$.
(Our sequence $\sigma$ needed for Proposition~\ref{finishing-prop} exists because $H$ is $3$-degenerate, since $H-v$ is outerplanar.)
\end{proof}

For convenience, we record the reducible configurations from Lemma~\ref{second-reduc-lem} in 
Figure~\ref{main-reduc-fig}(\subref{5165a-fig},%
\subref{67a-fig},\subref{6771a-fig}), along with all the configurations that we handle in our next lemma.

\bigskip

The argument proving Lemmas~\ref{first-reduc-lem} and~\ref{second-reduc-lem} 
works much more generally, and we will use it (in a more general form) to show that 
many more configurations are $k$-reducible for $k\ge 828$.  For these remaining configurations our proofs are more terse, 
closer to just a certificate of reducibility. 

\begin{lem}
Each of (RC\subref{53a-fig})--(RC\subref{546671b-fig}), shown in Figure~\ref{main-reduc-fig}, is $k$-reducible for all $k\ge 828$.
And if we decrease one or more vertex degrees in any of these configurations, the result is again $k$-reducible.
\label{main-reduc-lem}
\end{lem}
\begin{proof}
For a connected subgraph $H$ to be reducible, we require (a) that
\begin{align}
d_H(v) \ge 2d_G(v) - 9
\label{reduc-ineq}
\tag{$\star$}
\end{align}
for all vertices $v\in V(H)$.
The $H$ that we show reducible will have $d_G(v)\in\{5,6,7\}$ for each $v\in V(H)$.  We also require (b) that there exists some 
vertex $v$ such that inequality~\eqref{reduc-ineq} holds with strict inequality.  Finally, we require (c) that there exists an
order $\sigma$, as in Proposition~\ref{finishing-prop}, of $V(H)$ such that $d_G(v)+d_{\sigma}(v)\le 9$ for all $v\in V(H)$.
(Typically, $\sigma$ will be a 2-degeneracy order for $H$.) Every such subgraph $H$ will be $k$-reducible for sufficiently large 
$k$.  We highlight that these criteria are easy to quickly check by hand.  So it is simple for a person to identify reducible 
configurations while constructing a discharging proof.
The exact value of $k$ is not so important; the key conclusion of our proof will be that $G$ has a recoloring sequence of length 
$O(|G|)$.  However, as mentioned above we will show that (for all the configurations in Figure~\ref{main-reduc-fig}) 
we can let $k=828$.  Near the end of this section, we explain how we can reduce this to $k=520$.

A \Emph{well}, in a subgraph $H$ of $G$, is a vertex $v$ such that the inequality in~\eqref{reduc-ineq} is strict.  
To prove that a configuration
is reducible, we exhibit a directed acyclic graph in which each component is an out-tree (a tree directed outward from a root) that
is rooted at some well.  Intuitively, the well provides support to each other vertex in its out-tree.  As the diameter (and order)
of the out-tree $\vec{T}$ increase, so does the smallest $k$ for which $\vec{T}$ can be used to prove a configuration is 
$k$-reducible.

We write $r_s$ to denote a vertex of degree $r$ in $G$ and degree $s$ in $H$.  Our wells will typically be $5_2$-vertices and
$6_4$-vertices, although we occasionally use $5_3$-, $5_4$-, $5_5$-, $6_5$-, or $6_6$-vertices.  For example, 
Figure~\ref{outtree-fig}(4) shows a $6_4$-vertex with four $6_3$-neighbors, two of which each have their own $5_1$-neighbors.
Our calculations will consider together all the 
vertices in a single out-tree, showing that each is recolored at most $k$ times.

When an out-tree has an isomorphism mapping two vertices to each other, the calculations we perform for each of these two 
vertices will be identical.  So we typically label them with the same letter, and different subscripts, to highlight this;
and we only perform the calculations once.

We now proceed with the calculations for four specific out-trees, shown in Figure~\ref{outtree-fig}; we call the set of these $\Jc$. 
We say a bit more later about the larger set of out-trees $\prune(\Jc)$ that are also handled by our proofs for these four.  
But, for concreteness, it is best to now handle these four.  The calculations below follow the lines of those in the proof 
of Lemma~\ref{first-reduc-lem}, so we direct the reader there for more details.

As we have done previously, we denote each $5$-vertex by a square and each $6$-vertex by a circle. 
Pendent half-edges denote edges to vertices in subgraph $H$, but do not specify the other endpoints of these edges.
We list most of our reducible configurations in Figure~\ref{main-reduc-fig}. 
To prove reducibility for a configuration $H$ it suffices to cover the vertices of $H$ with disjoint out-trees in $\prune(\Jc)$.  

(1) We have $w_1$ defer to $x_1$ for the first $40$ colors that appear in $N_G(x_1)\setminus N_H(x_1)$.
So the number of times that $x_1$ is recolored is at most $\lceil (4k-40)/4\rceil+(40/5)+2=k$.
We have $v$ defer to $w_1$ for the first $144$ colors that appear in $N_G(w_1)\setminus N_H(w_1)$.
So the number of times that $w_1$ is recolored is at most $\lceil (3k-144)/3\rceil+((144+(40/5))/4)+(40/5)+2=k$.
The calculations for $w_2$ and $x_2$ are identical to those for $w_1$ and $x_1$.
Finally, the number of times that $v$ is recolored is at most $\lceil (3k+152/4+152/4)/4\rceil + 152/4 + 152/4 + 2$.
This expression is at most $k$ when $k\ge 4(38+38+2)+38+38=388$.

(2)
The calculations for $x$ and $y$ are identical to those for $w_1$ and $x_1$ in (1).  
Recall that $w$ may be recolored up to $152/4=38$ times due to $x$.
We have $v$ defer to $w$ for the first $594$ colors that appear in $N_G(w)\setminus N_H(w)$.
So the number of times that $w$ is recolored is at most 
$\lceil (3k-594)/3)\rceil+((594+38)/4)+38+2=k$.
We also have $v$ defer to $z$ for the first $24$ colors that appear in  $N_G(z)\setminus N_H(z)$.
So the number of times that $z$ is recolored is at most $\lceil (3k-24)/3\rceil+24/4+2=k$.
Thus, the number of times that $v$ is recolored is at most $\lceil (3k+158+6)/4\rceil+158+6+2$.
This expression is at most $k$ when $k\ge 4(41+158+6+2)=828$.

\begin{figure}[!h]
\centering
\begin{tikzpicture}[yscale=.9]
\tikzset{mysquare/.style={regular polygon,regular polygon sides=4, inner sep=0}} 
\tikzstyle{5vert}=[mysquare, minimum size = 8.5pt, inner sep = 1pt,
outer sep = 0pt, draw, fill=white]
\tikzstyle{6vert}=[shape = circle, minimum size = 16.0pt, inner sep = 2pt,
outer sep = 0pt, draw, fill=white]
\tikzset{every node/.style=6vert}
\def\half{.5}

\begin{scope}[yshift=-.4in]
\draw (0,0) node[5vert, inner sep=2.5pt] (v) {\footnotesize{$v$}} 
--++ (1,.5) node (w1) {\footnotesize{$w_1$}} --++ (1,0) node[5vert] (x1) {\footnotesize{$x_1$}} (v) 
--++ (1,-.5) node (w2) {\footnotesize{$w_2$}} --++ (1,0) node[5vert] (x2) {\footnotesize{$x_2$}};
\draw (w1) --++ (0,-\half) (w2) --++ (0,-\half);
\draw (1,-1.5) node[draw=none] {\footnotesize{(i)}};
\begin{scope}[ultra thick,decoration={
    markings,
    mark=at position 0.7 with {\arrow{>}}}
    ] 
    \draw[postaction={decorate}] (v)--(w1);
    \draw[postaction={decorate}] (w1)--(x1);
    \draw[postaction={decorate}] (v)--(w2);
    \draw[postaction={decorate}] (w2)--(x2);
\end{scope}
\end{scope}

\begin{scope}[yshift=-.4in]
\draw (0,0) node[5vert, inner sep=2.5pt] (v) {\footnotesize{$v$}} 
--++ (1,.5) node (w1) {\footnotesize{$w_1$}} --++ (1,0) node[5vert] (x1) {\footnotesize{$x_1$}} (v) 
--++ (1,-.5) node (w2) {\footnotesize{$w_2$}} --++ (1,0) node[5vert] (x2) {\footnotesize{$x_2$}};
\draw (w1) --++ (0,-\half) (w2) --++ (0,-\half);
\draw (1,-1.5) node[draw=none] {\footnotesize{(1)}};
\begin{scope}[ultra thick,decoration={
    markings,
    mark=at position 0.7 with {\arrow{>}}}
    ] 
    \draw[postaction={decorate}] (v)--(w1);
    \draw[postaction={decorate}] (w1)--(x1);
    \draw[postaction={decorate}] (v)--(w2);
    \draw[postaction={decorate}] (w2)--(x2);
\end{scope}
\end{scope}

\begin{scope}[xshift=1.5in, yshift=-0.4in]
\draw (0,0) node[5vert, inner sep=2.5pt] (v) {\footnotesize{$v$}} --++ (1,.5) node (w) {\footnotesize{$w$}} --++ (1,0) node (x) {\footnotesize{$x$}}
--++ (1,0) node[5vert, inner sep=2.5pt] (y) {\footnotesize{$y$}}
(v) --++ (1,-.5) node (z) {\footnotesize{$z$}};
\draw (w) --++ (0,-\half) (x) --++ (0,-\half);
\draw (z) --++ (0,-\half) (z) --++ (\half,0);
\draw (1.5,-1.5) node[draw=none] {\footnotesize{(2)}};
\begin{scope}[ultra thick,decoration={
    markings,
    mark=at position 0.7 with {\arrow{>}}}
    ] 
    \draw[postaction={decorate}] (v)--(w);
    \draw[postaction={decorate}] (w)--(x);
    \draw[postaction={decorate}] (x)--(y);
    \draw[postaction={decorate}] (v)--(z);
\end{scope}
\end{scope}

\begin{scope}[xshift=3.25in, yshift=-0.2in]
\draw (0,0) node[5vert, inner sep=2.5pt] (v) {\footnotesize{$v$}} 
--++ (1,1) node (w1) {\footnotesize{$w_1$}} --++ (1,0) node[5vert] (x1) {\footnotesize{$x_1$}} 
(v) --++ (1,0) node (w2) {\footnotesize{$w_2$}} --++ (1,0) node[5vert] (x2) {\footnotesize{$x_2$}}
(v) --++ (1,-1) node (y) {\footnotesize{$y$}};
\draw (w1) --++ (0,-\half) (w2) --++ (0,-\half);
\draw (y) --++ (0,-\half) (y) --++ (\half,0);
\draw (1,-2.0) node[draw=none] {\footnotesize{(3)}};
\begin{scope}[ultra thick,decoration={
    markings,
    mark=at position 0.7 with {\arrow{>}}}
    ] 
    \draw[postaction={decorate}] (v)--(w1);
    \draw[postaction={decorate}] (w1)--(x1);
    \draw[postaction={decorate}] (v)--(w2);
    \draw[postaction={decorate}] (w2)--(x2);
    \draw[postaction={decorate}] (v)--(y);
\end{scope}
\end{scope}

\begin{scope}[xshift=5in, yshift=0in]
\draw (0,0) node (v) {\footnotesize{$v$}} 
--++ (1,1.5) node (w1) {\footnotesize{$w_1$}} --++ (1,0) node[5vert] (x1) {\footnotesize{$x_1$}} 
(v) --++ (1,.5) node (w2) {\footnotesize{$w_2$}} --++ (1,0) node[5vert] (x2) {\footnotesize{$x_2$}}
(v) --++ (1,-.5) node (y1) {\footnotesize{$y_1$}}
(v) --++ (1,-1.5) node (y2) {\footnotesize{$y_2$}};
\draw (w1) --++ (0,-\half) (w2) --++ (0,-\half);
\draw (y1) --++ (0,-\half) (y1) --++ (\half,0);
\draw (y2) --++ (0,-\half) (y2) --++ (\half,0);
\draw (1,-2.5) node[draw=none] {\footnotesize{(4)}};
\begin{scope}[ultra thick,decoration={
    markings,
    mark=at position 0.7 with {\arrow{>}}}
    ] 
    \draw[postaction={decorate}] (v)--(w1);
    \draw[postaction={decorate}] (w1)--(x1);
    \draw[postaction={decorate}] (v)--(w2);
    \draw[postaction={decorate}] (w2)--(x2);
    \draw[postaction={decorate}] (v)--(y1);
    \draw[postaction={decorate}] (v)--(y2);
\end{scope}
\end{scope}
\end{tikzpicture}
\caption{The 4 configurations in the proof of Lemma~\ref{main-reduc-lem}\label{outtree-fig}}
\end{figure}

(3)
The calculations are the same as in (1) for $w_1,w_2,x_1,x_2$.  And the calculation for $y$ is the same as for $z$
in (2).  Thus, the number of times that $v$ is recolored is at most $\lceil (2k+38+38+6)/4\rceil+38+38+6+2$.
This expression is at most $k$ when $k\ge 2(38+38+6+2)+(38+38+6)/2=209$.

(4)
The calculations are the same as in (3) 
for $w_1,w_2,x_1,x_2,y_1,y_2$.  
Thus, the number of times that $v$ is recolored is at most $\lceil (2k+38+38+6+6)/3\rceil+38+38+6+6+2$.
This expression is at most $k$ when $k\ge 3(38+38+6+6+2)+38+38+6+6=358$.
\bigskip

\def\sf{.65}
\begin{figure}[!t]
\begin{subfigure}[b]{0.065\linewidth}
\begin{tikzpicture}[yscale=.866, xscale=.5, scale=\sf]
\draw (0,0) node[5vert] (a) {$~$} -- (1,1) node[5vert] (b) {$~$} (a) -- (2,0) node[5vert] (c) {$~$};
%
\begin{scope}[ultra thick,decoration={
    markings,
    mark=at position 0.6 with {\arrow{>}}}
    ] 
    \draw[postaction={decorate}] (a)--(b);
    \draw[postaction={decorate}] (a)--(c);
\end{scope}
\end{tikzpicture}
\subcaption{\label{53a-fig}}
\end{subfigure}
\begin{subfigure}[b]{0.085\linewidth}
\begin{tikzpicture}[yscale=.866, xscale=.5, scale=\sf]
\draw (0,0) node[5vert] (a) {$~$} -- (1,1) node (b) {} -- (2,0) node (c) {} -- (3,1) node[5vert] (d) {$~$};
%
\begin{scope}[ultra thick,decoration={
    markings,
    mark=at position 0.6 with {\arrow{>}}}
    ] 
    \draw[postaction={decorate}] (a)--(c);
    \draw[postaction={decorate}] (d)--(b);
\end{scope}
\end{tikzpicture}
\subcaption{\label{5262a-fig}}
\end{subfigure}
\begin{subfigure}[b]{0.1\linewidth}
\begin{tikzpicture}[yscale=.866, xscale=.5, scale=\sf]
\draw (0,0) node[5vert] (a) {$~$} -- (1,1) node (b) {} -- (2,0) node (c) {} -- (3,1) node (d) {} -- (4,0) node[5vert] (e) {$~$};

\begin{scope}[ultra thick,decoration={
    markings,
    mark=at position 0.6 with {\arrow{>}}}
    ] 
    \draw[postaction={decorate}] (a)--(c);
    \draw[postaction={decorate}] (a)--(b);
    \draw[postaction={decorate}] (b)--(d);
    \draw[postaction={decorate}] (d)--(e);
\end{scope}
\end{tikzpicture}
\subcaption{\label{5263a-fig}}
\end{subfigure}
~%
\begin{subfigure}[b]{0.12\linewidth}
\begin{tikzpicture}[yscale=.866, xscale=.5, scale=\sf]
\draw (0,0) node[5vert] (a) {$~$} -- (1,1) node (b) {} -- (2,0) node (c) {} -- (3,1) node (d) {} -- (4,0) node (e) {} --  (5,1) node[5vert] (f) {$~$} --(e) --(c) -- (a) (b) -- (d);

\begin{scope}[ultra thick,decoration={
    markings,
    mark=at position 0.7 with {\arrow{>}}}
    ] 
    \draw[postaction={decorate}] (a)--(b);
    \draw[postaction={decorate}] (c)--(d);
    \draw[postaction={decorate}] (c)--(e);
    \draw[postaction={decorate}] (e)--(f);
\end{scope}
\end{tikzpicture}
\subcaption{\label{5264a-fig}}
\end{subfigure}
~%
\begin{subfigure}[b]{0.15\linewidth}
\begin{tikzpicture}[yscale=.866, xscale=.5, scale=\sf]
\draw (0,0) node[5vert] (a) {$~$} -- (1,1) node (b) {} -- (2,0) node (c) {} -- (3,1) node (d) {} -- (4,0) node (e) {} --  (5,1) node (f) {} -- (6,0) node[5vert] (g) {$~$} (f)--(e) --(c) (b) -- (d) -- (f);

\begin{scope}[ultra thick,decoration={
    markings,
    mark=at position 0.7 with {\arrow{>}}}
    ] 
    \draw[postaction={decorate}] (d)--(b);
    \draw[postaction={decorate}] (b)--(a);
    \draw[postaction={decorate}] (d)--(c);
    \draw[postaction={decorate}] (d)--(e);
    \draw[postaction={decorate}] (d)--(f);
    \draw[postaction={decorate}] (f)--(g);
\end{scope}

\end{tikzpicture}
\subcaption{\label{5265a-fig}~}
\end{subfigure}
\begin{subfigure}[b]{0.19\linewidth}
\begin{tikzpicture}[yscale=.866, xscale=.5, scale=\sf]
\draw (0,0) node[5vert] (a) {$~$} -- (1,1) node (b) {} -- (2,0) node (c) {} -- (3,1) node (d) {} -- (4,0) node (e) {} --  (5,1) node (f) {} -- (6,0) node (g) {} -- (8,0) node[5vert] (i) {$~$} (f) -- (7,1) node[5vert] (h) {$~$} 
(f)--(e) --(c) -- (a) (b) -- (d) (e) -- (g);

\begin{scope}[ultra thick,decoration={
    markings,
    mark=at position 0.7 with {\arrow{>}}}
    ] 
    \draw[postaction={decorate}] (a)--(b);
    \draw[postaction={decorate}] (c)--(d);
    \draw[postaction={decorate}] (e)--(f);
    \draw[postaction={decorate}] (f)--(h);
    \draw[postaction={decorate}] (e)--(g);
    \draw[postaction={decorate}] (g)--(i);
\end{scope}
\end{tikzpicture}
\subcaption{\label{5366a-fig}~}
\end{subfigure}
\begin{subfigure}[b]{0.185\linewidth}
\begin{tikzpicture}[yscale=.866, xscale=.5, scale=\sf]
\draw (0,0) node[5vert] (a) {$~$} -- (1,1) node (b) {} -- (2,0) node (c) {} -- (3,1) node (d) {} -- (4,0) node (e) {} --  (5,1) node (f) {} -- (6,0) node (g) {} -- (8,0) node[5vert] (i) {$~$} (f) 
(f)--(e) --(c) -- (a) (b) -- (d) -- (f) (e) -- (g);

\begin{scope}[ultra thick,decoration={
    markings,
    mark=at position 0.7 with {\arrow{>}}}
    ] 
    \draw[postaction={decorate}] (a)--(b);
    \draw[postaction={decorate}] (c)--(d);
    \draw[postaction={decorate}] (e)--(f);
    \draw[postaction={decorate}] (e)--(g);
    \draw[postaction={decorate}] (g)--(i);
\end{scope}
\end{tikzpicture}
\subcaption{\label{5266a-fig}}
\end{subfigure}
\end{figure}
\begin{figure}[!t]
\ContinuedFloat
\begin{subfigure}[b]{0.095\linewidth}
\begin{tikzpicture}[yscale=.866, xscale=.5, scale=\sf]
\draw (0,0) node[5vert] (a) {$~$} -- (1,1) node (b) {} -- (2,0) node[5vert] (c) {$~$} -- (3,1) node (d) {} -- (4,0) node[5vert] (e) {$~$}
(b) -- (d);

\begin{scope}[ultra thick,decoration={
    markings,
    mark=at position 0.7 with {\arrow{>}}}
    ] 
    \draw[postaction={decorate}] (c)--(b);
    \draw[postaction={decorate}] (b)--(a);
    \draw[postaction={decorate}] (c)--(d);
    \draw[postaction={decorate}] (d)--(e);
\end{scope}
\end{tikzpicture}
\subcaption{\label{5362a-fig}}
\end{subfigure}
~~%
\begin{subfigure}[b]{0.125\linewidth}
\begin{tikzpicture}[yscale=.866, xscale=.5, scale=\sf]
\draw (0,0) node[5vert] (a) {$~$} -- (1,1) node (b) {} -- (2,0) node[5vert] (c) {$~$} -- (3,1) node (d) {} -- (4,0) node (e) {} --  (5,1) node[5vert] (f) {$~$} --(e) --(c) (b) -- (d);

\begin{scope}[ultra thick,decoration={
    markings,
    mark=at position 0.7 with {\arrow{>}}}
    ] 
    \draw[postaction={decorate}] (c)--(b);
    \draw[postaction={decorate}] (b)--(a);
    \draw[postaction={decorate}] (c)--(d);
    \draw[postaction={decorate}] (c)--(e);
    \draw[postaction={decorate}] (e)--(f);
\end{scope}
\end{tikzpicture}
\subcaption{\label{5363a-fig}}
\end{subfigure}
\begin{subfigure}[b]{0.17\linewidth}
\begin{tikzpicture}[yscale=.866, xscale=.5, scale=\sf]
\draw (0,0) node[5vert] (a) {$~$} -- (1,1) node (b) {} -- (2,0) node (c) {} -- (3,1) node (d) {} -- (4,0) node (e) {} --  (5,1) node (f) {} (6,0) node[5vert] (g) {$~$} (f) -- (7,1) node[5vert] (h) {$~$} 
(f)--(e) (c) -- (a) (b) -- (d) -- (f) (e) -- (g);

\begin{scope}[ultra thick,decoration={
    markings,
    mark=at position 0.7 with {\arrow{>}}}
    ] 
    \draw[postaction={decorate}] (a)--(b);
    \draw[postaction={decorate}] (a)--(c);
    \draw[postaction={decorate}] (d)--(e);
    \draw[postaction={decorate}] (e)--(g);
    \draw[postaction={decorate}] (d)--(f);
    \draw[postaction={decorate}] (f)--(h);
\end{scope}
\end{tikzpicture}
\subcaption{\label{5365a-fig}}
\end{subfigure}
%
\begin{subfigure}[b]{0.17\linewidth}
\begin{tikzpicture}[yscale=.866, xscale=.5, scale=\sf]
\draw (0,1) node[5vert] (a) {$~$}  -- (1,0) node (b) {} -- (2,1) node (c) {} -- (3,0) node (d) {} -- (4,1) node (e) {} -- (5,0) node (f) {} -- (6,1) node[5vert] (g) {$~$} (f) -- (7,0) node[5vert] (h) {$~$}
(b) -- (d) (c) -- (e) -- (g);

\begin{scope}[ultra thick,decoration={
    markings,
    mark=at position 0.7 with {\arrow{>}}}
    ] 
    \draw[postaction={decorate}] (e)--(c);
    \draw[postaction={decorate}] (e)--(d);
    \draw[postaction={decorate}] (d)--(b);
    \draw[postaction={decorate}] (b)--(a);
    \draw[postaction={decorate}] (g)--(f);
    \draw[postaction={decorate}] (f)--(h);
\end{scope}
\end{tikzpicture}
\subcaption{\label{5365b-fig}}
\end{subfigure}
\begin{subfigure}[b]{0.17\linewidth}
\begin{tikzpicture}[yscale=.866, xscale=.5, scale=\sf]
\draw (0,0) node[5vert] (a) {$~$} (1,1) node[5vert] (b) {$~$} -- (2,0) node (c) {} -- (3,1) node (d) {} -- (4,0) node (e) {} --  (5,1) node (f) {} (6,0) node[5vert] (g) {$~$} (f) -- (7,1) node[5vert] (h) {$~$} 
(f)--(e) (c) -- (a) (b) -- (d) -- (f) (e) -- (g);

\begin{scope}[ultra thick,decoration={
    markings,
    mark=at position 0.7 with {\arrow{>}}}
    ] 
    \draw[postaction={decorate}] (b)--(c);
    \draw[postaction={decorate}] (c)--(a);
    \draw[postaction={decorate}] (d)--(e);
    \draw[postaction={decorate}] (e)--(g);
    \draw[postaction={decorate}] (d)--(f);
    \draw[postaction={decorate}] (f)--(h);
\end{scope}
\end{tikzpicture}
\subcaption{\label{5464a-fig}}
\end{subfigure}
%
\begin{subfigure}[b]{0.185\linewidth}
\begin{tikzpicture}[yscale=.866, xscale=.5, scale=\sf]
\draw (-1,0) node[5vert] (aa) {$~$} -- (0,1) node (a) {} -- (1,0) node (b) {} -- (2,1) node (c) {} -- (3,0) node (d) {} -- (4,1) node (e) {} -- (5,0) node (f) {} -- (6,1) node[5vert] (g) {$~$} (f) -- (7,0) node[5vert] (h) {$~$}
(b) -- (d) (c) -- (e) -- (g);

\begin{scope}[ultra thick,decoration={
    markings,
    mark=at position 0.7 with {\arrow{>}}}
    ] 
    \draw[postaction={decorate}] (e)--(d);
    \draw[postaction={decorate}] (d)--(b);
    \draw[postaction={decorate}] (c)--(a);
    \draw[postaction={decorate}] (a)--(aa);
    \draw[postaction={decorate}] (g)--(f);
    \draw[postaction={decorate}] (f)--(h);
\end{scope}
\end{tikzpicture}
\subcaption{\label{5366b-fig}}
\end{subfigure}
\bigskip

\begin{subfigure}[b]{0.11\linewidth}
\begin{tikzpicture}[yscale=-.866, xscale=.5, scale=\sf]
\draw (-2,0) node (a) {} --++ (1,1) node (b) {} --++ (2,0) node (c) {} --++ (1,-1) node (d) {}
--++ (-2,-1) node (e) {} -- (a) (0,0) node[5vert] (h) {$~$};

\begin{scope}[ultra thick,decoration={
    markings,
    mark=at position 0.7 with {\arrow{>}}}
    ] 
    \draw[postaction={decorate}] (h)--(a);
    \draw[postaction={decorate}] (h)--(b);
    \draw[postaction={decorate}] (h)--(c);
    \draw[postaction={decorate}] (h)--(d);
    \draw[postaction={decorate}] (h)--(e);
\end{scope}
\end{tikzpicture}
\subcaption{\label{5165a-fig}}
\end{subfigure}
%
%
%
%
\begin{subfigure}[b]{0.13\linewidth}
\begin{tikzpicture}[yscale=.866, xscale=.5, scale=\sf]
\draw (1,-1) node[5vert] (b) {$~$}  -- (2,0) node (c) {} -- (3,-1) node (d) {} -- (4,0) node (e) {} --  (5,-1) node (f) {} (6,0) node (g) {} -- (5,1) node (h) {} -- (3,1) node[5vert] (i) {$~$} 
(f)--(e) -- (c) (b) -- (d) -- (f) -- (g) -- (e);

\begin{scope}[ultra thick,decoration={
    markings,
    mark=at position 0.7 with {\arrow{>}}}
    ] 
    \draw[postaction={decorate}] (e)--(c);
    \draw[postaction={decorate}] (e)--(f);
    \draw[postaction={decorate}] (e)--(g);
    \draw[postaction={decorate}] (e)--(h);
    \draw[postaction={decorate}] (h)--(i);
\end{scope}
\end{tikzpicture}
\subcaption{\label{5266b-fig}}
\end{subfigure}
\begin{subfigure}[b]{0.15\linewidth}
\begin{tikzpicture}[yscale=.866, xscale=.5, scale=\sf]
\draw (0,0) node[5vert] (a) {$~$} -- (1,-1) node (b) {} -- (2,0) node (c) {} -- (3,-1) node (d) {} -- (4,0) node (e) {} --  (5,-1) node (f) {} (6,0) node (g) {} -- (5,1) node (h) {} -- (3,1) node[5vert] (i) {$~$} 
(f)--(e) -- (c) -- (a) (b) -- (d) -- (f) -- (g) -- (e);

\begin{scope}[ultra thick,decoration={
    markings,
    mark=at position 0.7 with {\arrow{>}}}
    ] 
    \draw[postaction={decorate}] (a)--(b);
    \draw[postaction={decorate}] (e)--(h);
    \draw[postaction={decorate}] (e)--(f);
    \draw[postaction={decorate}] (e)--(g);
    \draw[postaction={decorate}] (h)--(i);
\end{scope}
\end{tikzpicture}
\subcaption{\label{5267a-fig}~}
\end{subfigure}
%
%
\begin{subfigure}[b]{0.15\linewidth}
\begin{tikzpicture}[yscale=.866, xscale=.5, scale=\sf]
\draw (2,0) node[5vert] (c) {$~$} -- (3,-1) node (d) {} -- (4,0) node (e) {} --  (5,-1) node (f) {} (6,0) node (g) {} -- (5,1) node (h) {} -- (7,1) node (i) {} -- (8,0) node (j) {} -- (7,-1) node[5vert] (k) {$~$} 
(g) -- (f)--(e) -- (c) (d) -- (f) (e) -- (g) -- (i);

\begin{scope}[ultra thick,decoration={
    markings,
    mark=at position 0.7 with {\arrow{>}}}
    ] 
    \draw[postaction={decorate}] (e)--(d);
    \draw[postaction={decorate}] (e)--(f);
    \draw[postaction={decorate}] (e)--(h);
    \draw[postaction={decorate}] (h)--(i);
    \draw[postaction={decorate}] (g)--(j);
    \draw[postaction={decorate}] (j)--(k);
\end{scope}
\end{tikzpicture}
\subcaption{\label{5267b-fig}}
\end{subfigure}
\begin{subfigure}[b]{0.15\linewidth}
\begin{tikzpicture}[yscale=.866, xscale=.5, scale=\sf]
\draw (0,0) node (a) {} -- (1,1) node[5vert] (b) {$~$} -- (2,0) node (c) {} -- (1,-1) node (d) {} -- (3,-1) node (e) {} -- (c) -- (4,0) node (f) {} -- (5,-1) node (g) {} -- (6,0) node (h) {} -- (5,1) node[5vert] (i) {$~$} (f)
(d) -- (a) -- (c) (h) -- (f) -- (e) -- (g);

\begin{scope}[ultra thick,decoration={
    markings,
    mark=at position 0.7 with {\arrow{>}}}
    ] 
    \draw[postaction={decorate}] (c)--(a);
    \draw[postaction={decorate}] (c)--(d);
    \draw[postaction={decorate}] (f)--(g);
    \draw[postaction={decorate}] (f)--(h);
    \draw[postaction={decorate}] (h)--(i);
\end{scope}
\end{tikzpicture}
\subcaption{\label{5267c-fig}}
\end{subfigure}
%
\begin{subfigure}[b]{0.17\linewidth}
\begin{tikzpicture}[yscale=.866, xscale=.5, scale=\sf]
\draw (0,0) node (a) {} -- (1,1) node[5vert] (b) {$~$} -- (2,0) node (c) {} -- (1,-1) node (d) {} -- (3,-1) node (e) {} -- (c) -- (4,0) node (f) {} -- (5,-1) node (g) {} -- (6,0) node (h) {} -- (5,1) node (i) {} -- (7,1) node[5vert] (j) {$~$} (f)
(d) -- (a) -- (c) (f) -- (e) -- (g) -- (h);

\begin{scope}[ultra thick,decoration={
    markings,
    mark=at position 0.7 with {\arrow{>}}}
    ] 
    \draw[postaction={decorate}] (c)--(a);
    \draw[postaction={decorate}] (c)--(d);
    \draw[postaction={decorate}] (f)--(h);
    \draw[postaction={decorate}] (f)--(g);
    \draw[postaction={decorate}] (f)--(i);
    \draw[postaction={decorate}] (i)--(j);
\end{scope}
\end{tikzpicture}
\subcaption{\label{5268b-fig}~~~~}
\end{subfigure}
\bigskip

%
\begin{subfigure}[b]{0.1925\linewidth}
\begin{tikzpicture}[yscale=.866, xscale=.5, scale=\sf]
\draw (0,0) node (a) {} -- (1,1) node[5vert] (b) {$~$} -- (2,0) node (c) {} -- (1,-1) node (d) {} -- (3,-1) node (e) {} -- (c) -- (4,0) node (f) {} -- (5,-1) node (g) {} -- (7,-1) node[5vert] (h) {$~$} (f) -- (6,0) node (i) {} -- (8,0) node[5vert] (j) {$~$} 
(d) -- (a) -- (c) (e) -- (f) (g) -- (i);

\begin{scope}[ultra thick,decoration={
    markings,
    mark=at position 0.7 with {\arrow{>}}}
    ] 
    \draw[postaction={decorate}] (c)--(a);
    \draw[postaction={decorate}] (c)--(d);
    \draw[postaction={decorate}] (c)--(e);
    \draw[postaction={decorate}] (f)--(g);
    \draw[postaction={decorate}] (g)--(h);
    \draw[postaction={decorate}] (f)--(i);
    \draw[postaction={decorate}] (i)--(j);
\end{scope}
\end{tikzpicture}
\subcaption{\label{5367a-fig}}
\end{subfigure}
%
%
%
%
\begin{subfigure}[b]{0.17\linewidth}
\begin{tikzpicture}[yscale=.866, xscale=.5, scale=\sf]
\draw (3,2) node[5vert] (aaa) {$~$}-- (1,2) node(aa) {} -- (0,1) node (a) {} -- (1,0) node (b) {} -- (2,1) node (c) {} -- (3,0) node (d) {} -- (4,1) node (e) {} -- (5,0) node (f) {} -- (6,1) node[5vert] (g) {$~$} (f) -- (7,0) node[5vert] (h) {$~$}
(b) -- (d) (c) -- (e) -- (g) (c) -- (aa);

\begin{scope}[ultra thick,decoration={
    markings,
    mark=at position 0.7 with {\arrow{>}}}
    ] 
    \draw[postaction={decorate}] (c)--(d);
    \draw[postaction={decorate}] (c)--(b);
    \draw[postaction={decorate}] (c)--(a);
    \draw[postaction={decorate}] (c)--(aa);
    \draw[postaction={decorate}] (aa)--(aaa);
    \draw[postaction={decorate}] (e)--(f);
    \draw[postaction={decorate}] (f)--(h);
\end{scope}
\end{tikzpicture}
\subcaption{\label{5367b-fig}}
\end{subfigure}
%
\begin{subfigure}[b]{0.15\linewidth}
\begin{tikzpicture}[yscale=.866, xscale=.5, scale=\sf]
\draw (3,1) node (aa) {} (0,0) node (a) {} -- (1,1) node (b) {} -- (2,0) node (c) {} -- (1,-1) node (d) {} -- (3,-1) node (e) {} -- (c) -- (4,0) node[7vert] (f) {} -- (5,-1) node (g) {} -- (6,0) node[5vert] (h) {$~$} 
(d) -- (a) -- (c) (h) -- (f) -- (e) -- (g) (b) -- (aa) -- (c) (aa) -- (f);

\begin{scope}[ultra thick,decoration={
    markings,
    mark=at position 0.7 with {\arrow{>}}}
    ] 
    \draw[postaction={decorate}] (c)--(a);
    \draw[postaction={decorate}] (c)--(aa);
    \draw[postaction={decorate}] (c)--(b);
    \draw[postaction={decorate}] (c)--(d);
    \draw[postaction={decorate}] (e)--(f);
    \draw[postaction={decorate}] (e)--(g);
    \draw[postaction={decorate}] (g)--(h);
\end{scope}
\end{tikzpicture}
\subcaption{\label{516771a-fig}}
\end{subfigure}
\begin{subfigure}[b]{0.17\linewidth}
\begin{tikzpicture}[yscale=.866, xscale=.5, scale=\sf]
\draw (0,0) node (a) {} -- (1,1) node[5vert] (b) {$~$} -- (2,0) node[7vert] (c) {$\,$} -- (1,-1) node (d) {} -- (3,-1) node (e) {} -- (c) -- (4,0) node (f) {} -- (5,-1) node (g) {} -- (7,-1) node[5vert] (h) {$~$} (f)
(a) -- (c) (f) -- (e) -- (g);

\begin{scope}[ultra thick,decoration={
    markings,
    mark=at position 0.7 with {\arrow{>}}}
    ] 
    \draw[postaction={decorate}] (b)--(a);
    \draw[postaction={decorate}] (b)--(c);
    \draw[postaction={decorate}] (a)--(d);
    \draw[postaction={decorate}] (e)--(f);
    \draw[postaction={decorate}] (e)--(g);
    \draw[postaction={decorate}] (g)--(h);
\end{scope}
\end{tikzpicture}
\subcaption{\label{526571a-fig}}
\end{subfigure}
%
%
%
\begin{subfigure}[b]{0.17\linewidth}
\begin{tikzpicture}[yscale=.866, xscale=.5, scale=\sf]
\draw (0,0) node (a) {} -- (1,1) node (b) {} -- (3,1) node[5vert] (aa) {$~$} (b) -- (2,0) node[7vert] (c) {$\,$} -- (1,-1) node[5vert] (d) {$~$} -- (3,-1) node (e) {} -- (c) -- (4,0) node (f) {} -- (5,-1) node (g) {} -- (7,-1) node[5vert] (h) {$~$} (f)
(a) -- (c) (f) -- (e) -- (g);

\begin{scope}[ultra thick,decoration={
    markings,
    mark=at position 0.7 with {\arrow{>}}}
    ] 
    \draw[postaction={decorate}] (d)--(a);
    \draw[postaction={decorate}] (a)--(b);
    \draw[postaction={decorate}] (b)--(aa);
    \draw[postaction={decorate}] (d)--(c);
    \draw[postaction={decorate}] (e)--(f);
    \draw[postaction={decorate}] (e)--(g);
    \draw[postaction={decorate}] (g)--(h);
\end{scope}
\end{tikzpicture}
\subcaption{\label{536571a-fig}}
\end{subfigure}
\end{figure}
\begin{figure}[!ht]
\ContinuedFloat

\begin{subfigure}[b]{0.12\linewidth}
\begin{tikzpicture}[yscale=.866, xscale=.5, scale=\sf]
\draw (0,0) node (a) {} (-2,0) node (b) {} -- (-1,1) node (c) {} -- (1,1) node (d) {} -- (2,0) node (e) {} -- (1,-1) node (f) {}
-- (-1,-1) node (g) {} -- (b);

\begin{scope}[ultra thick,decoration={
    markings,
    mark=at position 0.7 with {\arrow{>}}}
    ] 
    \draw[postaction={decorate}] (a)--(b);
    \draw[postaction={decorate}] (a)--(c);
    \draw[postaction={decorate}] (a)--(d);
    \draw[postaction={decorate}] (a)--(e);
    \draw[postaction={decorate}] (a)--(f);
    \draw[postaction={decorate}] (a)--(g);
\end{scope}
\end{tikzpicture}
\subcaption{\label{67a-fig}~~}
\end{subfigure}
\begin{subfigure}[b]{0.12\linewidth}
\begin{tikzpicture}[yscale=.866, xscale=.5, scale=\sf]
\draw (-2,0) node (a) {} --++ (1,1) node (b) {} --++ (2,0) node (c) {} --++ (1,-1) node (d) {}
--++ (-.6,-1) node (e) {} --++ (-1.4,0) node (f) {} --++ (-1.4,0) node (g) {} -- (a) (0,0) node[7vertB] (h) {};

\begin{scope}[ultra thick,decoration={
    markings,
    mark=at position 0.7 with {\arrow{>}}}
    ] 
    \draw[postaction={decorate}] (h)--(a);
    \draw[postaction={decorate}] (h)--(b);
    \draw[postaction={decorate}] (h)--(c);
    \draw[postaction={decorate}] (h)--(d);
    \draw[postaction={decorate}] (h)--(e);
    \draw[postaction={decorate}] (h)--(f);
    \draw[postaction={decorate}] (h)--(g);
\end{scope}
\end{tikzpicture}
\subcaption{\label{6771a-fig}~~~}
\end{subfigure}
%
%
\begin{subfigure}[b]{0.12\linewidth}
\begin{tikzpicture}[yscale=.866, xscale=.5, scale=\sf]
\draw (0,0) node (b) {} -- (1,-1) node (c) {} -- (3,-1) node (d) {} -- 
(2,0) node[7vert] (e) {$\,$} -- (1,1) node[5vert] (f) {$~$} (d) -- (4,0) node[5vert] (h) {$~$} -- (e) -- (c) (b) -- (e);

\begin{scope}[ultra thick,decoration={
    markings,
    mark=at position 0.7 with {\arrow{>}}}
    ] 
    \draw[postaction={decorate}] (f)--(b);
    \draw[postaction={decorate}] (f)--(e);
    \draw[postaction={decorate}] (h)--(d);
    \draw[postaction={decorate}] (d)--(c);
\end{scope}
\end{tikzpicture}
\subcaption{\label{526371a-fig}~~~}
\end{subfigure}
%
\begin{subfigure}[b]{0.14\linewidth}
\begin{tikzpicture}[yscale=.866, xscale=.5, scale=\sf]
\draw (-1,-1) node[5vert] (a) {$~$} -- (0,0) node (b) {} -- (1,-1) node (c) {} -- (3,-1) node[5vert] (d) {$~$} -- 
(2,0) node[7vert] (e) {$\,$} -- (1,1) node (f) {} -- (3,1) node (g) {} -- (4,0) node[5vert] (h) {$~$} (a) -- (c) -- (e) --(g);

\begin{scope}[ultra thick,decoration={
    markings,
    mark=at position 0.7 with {\arrow{>}}}
    ] 
    \draw[postaction={decorate}] (b)--(f);
    \draw[postaction={decorate}] (b)--(e);
    \draw[postaction={decorate}] (f)--(g);
    \draw[postaction={decorate}] (g)--(h);
\end{scope}
\end{tikzpicture}
\subcaption{\label{536471a-fig}}
\end{subfigure}
\begin{subfigure}[b]{0.2\linewidth}
\begin{tikzpicture}[yscale=.866, xscale=.5, scale=\sf]
\draw (-1,-1) node (a) {} -- (0,0) node[7vert] (b) {$\,$} -- (1,-1) node (c) {} -- (3,-1) node (d) {} -- 
(2,0) node (e) {} -- (1,1) node (f) {} -- (3,1) node[5vert]  (g) {$~$} (d) -- (5,-1) node[5vert] (i) {$~$} (a) -- (c) -- (e);
\draw (-3,-1) node[5vert] (aaa) {$~$} -- (-2,0) node (aa) {} -- (a) (aa) -- (b) -- (f);

\begin{scope}[ultra thick,decoration={
    markings,
    mark=at position 0.7 with {\arrow{>}}}
    ] 
    \draw[postaction={decorate}] (e)--(f);
    \draw[postaction={decorate}] (f)--(g);
    \draw[postaction={decorate}] (e)--(d);
    \draw[postaction={decorate}] (d)--(i);
    \draw[postaction={decorate}] (e)--(b);
    \draw[postaction={decorate}] (c)--(a);
    \draw[postaction={decorate}] (a)--(aa);
    \draw[postaction={decorate}] (aa)--(aaa);
\end{scope}
\end{tikzpicture}
\subcaption{\label{546571a-fig}~~}
\end{subfigure}
\begin{subfigure}[b]{0.16\linewidth}
\begin{tikzpicture}[yscale=.866, xscale=.5, scale=\sf]
\draw (-1,-1) node (a) {} -- (0,0) node[7vert] (b) {$\,$} -- (1,-1) node (c) {} -- (3,-1) node (d) {} -- 
(2,0) node (e) {} -- (1,1) node (f) {} -- (3,1) node (g) {} -- (4,0) node[5vert] (h) {$~$} (d) -- (5,-1) node[5vert] (i) {$~$} (a) -- (c) -- (e) --(g);
\draw (-3,-1) node[5vert] (aaa) {$~$} -- (-2,0) node (aa) {} -- (a) (aa) -- (b) -- (f);

\begin{scope}[ultra thick,decoration={
    markings,
    mark=at position 0.7 with {\arrow{>}}}
    ] 
    \draw[postaction={decorate}] (e)--(f);
    \draw[postaction={decorate}] (e)--(g);
    \draw[postaction={decorate}] (g)--(h);
    \draw[postaction={decorate}] (e)--(d);
    \draw[postaction={decorate}] (d)--(i);
    \draw[postaction={decorate}] (e)--(b);
    \draw[postaction={decorate}] (c)--(a);
    \draw[postaction={decorate}] (a)--(aa);
    \draw[postaction={decorate}] (aa)--(aaa);
\end{scope}
\end{tikzpicture}
\subcaption{\label{546671b-fig}}
\end{subfigure}
\caption{\subref{546671b-fig} configurations shown reducible by Lemmas~\ref{first-reduc-lem}, \ref{second-reduc-lem}, and~\ref{main-reduc-lem}\label{main-reduc-fig}}
\end{figure}

A key advantage of the approach we have taken above is that we can perform the calculations for distinct out-trees independently.  
What is more, we can reuse these same calculations to prove the reducibility of many configurations. 
It is easy to check that if a calculation works for an out-tree $\vec{T}$, then it also works for any out-tree $\vec{T'}$ that 
results by deleting a leaf from $\vec{T}$.
Similarly, we can replace any $r_s$-vertex in an out-tree with either an $r_{s+1}$-vertex or an $(r-1)_s$-vertex; both types 
of modification reduce the number of constraints on our recoloring sequence, so the calculations for the original out-tree
also certify the new out-tree.
It is also helpful to note that we can replace any $5_2$-vertex root of an out-tree with a $6_4$-vertex root (the computation
allowing this replacement essentially rests on the fact that $(6-4)/(10-6-1) = 2/3 < 3/4 = (5-2)/(10-5-1)$).
Furthermore, we can replace any $5_1$-vertex leaf in an out-tree with a $6_3$-vertex leaf, and we can replace any $6_3$-vertex leaf
with a $7_5$-vertex leaf (the validity of these replacements use the fact that $1/(3\times 2) > 1/(4\times 3) > 1/(5\times 4)$).
Finally, if a $6_3$-vertex supports only a single leaf $5_1$-vertex, then the $(6_3(5_1))$ can be replaced by a single $5_1$.
We call any tree resulting from a sequence of zero or more of any combination of these operations a \Emph{pruning} of $\vec{T}$;
we denote the set of prunings of $\Jc$ by \Emph{$\prune(\Jc)$}. 
(We show in Lemma~\ref{edge-remove-lem} that we can also remove an edge from any configuration in Figure~\ref{main-reduc-fig}, 
decreasing the degrees of its endpoints, and get another reducible configuration; further, we can iterate this operation.)
In this paper, all of our out-trees have radius at most 3, although in principle the same technique could be applied to out-trees
of arbitrary (but bounded) radius.

We note that our reducibility proof does not require that these configurations appear as induced subgraphs.  The argument works
equally well for subgraphs that are not induced.  The proofs also work when vertices at pairwise distance 3 or more are identified,
as long as all but one of the identified vertices is a leaf in its out-tree.  For example, if a $5_2$- and a $5_1$-vertex are 
identified, the result is a $5_3$-vertex.  So in the collection of out-trees witnessing the configuration's reducibility, we can 
simply prune the $5_1$-vertex from its out-tree.  (Note that each $5_1$-vertex is necessarily a leaf in its out-tree.)
One other case of possible vertex identification is worth noting.  Suppose that a configuration has two $6_3$-vertices, at least
one of which supports only a $5_1$-vertex.  If these vertices are identified, then we simply delete the $5_1$-vertex from 
the resulting configuration.  The identified vertex is now a $6_5$-vertex, and the $6_3$-vertex that was supporting the (now deleted)
$5_1$-vertex can be deleted from its out-tree, along with the deleted $5_1$-vertex.
\end{proof}

\subsection{One More Wrinkle}

The technique we introduced to prove Lemma~\ref{main-reduc-lem} is powerful.  However, to handle a few more reducible 
configurations, we will need to extend it.  Rather than stating this extension in a general form, we only state the
particular special cases that we need for our 4 remaining reducible configurations.  But the reader should find that 
generalizing it (for other problems) is straightforward.

\begin{lem}
If $G$ is a $k$-minimal counterexample, with $k\ge 520$, then $G$ has none of the configurations in Figure~\ref{more-reduc-fig}.
\label{more-reduc-lem}
\end{lem}

\begin{proof}
(\subref{5261a-fig}) Suppose the contrary.  Let $G$, $L$, $\a$, $\b$ be a counterexample, 
let $G':=G-\{v_1,v_2\}$, and let $G'':=G'-\{w\}$.  Let $\a'$ and $\b'$ denote the restrictions of $\a$ and $\b$ to $G'$, 
let $\a''$ and $\b''$ denote the restrictions of $\a$ and $\b$ to $G''$, and let $\sigma''$ denote a $k$-good reconfiguration
sequence from $\a''$ to $\b''$.  By the Extension Lemma, we can extend $\sigma''$ to a $k$-good reconfiguration 
sequence $\sigma'$ in $G'$ from $\a'$ to $\b'$.  Furthermore, $w$ is recolored at most $\ceil{4k/5}+1$ times. 

Now we extend $\sigma'$ to a $k$-good recoloring sequence from $\a$ to $\b$.  Neither $v_1$ nor $v_2$ defers to the other.
So the number of recolorings in $N(v_1)$ that might cause us to recolor $v_1$ is at most $3k+\ceil{4k/5}+1$.  Since 
$|L(v_1)|-d_G(v_1)-1=10-5-1=4$, the total number of recolorings of $v_1$ is at most $\ceil{(3k+\ceil{4k/5}+1)/4}+2\le k$.
The same argument works for $v_2$.

(\subref{5481a-fig}) The argument is nearly identical to that for (1): First we extend to $w$, then to $v_1$ and $v_2$, and 
finally to $v_3$ and $v_4$, with the same analysis as for $v_1$ and $v_2$.

For each of (\subref{526271a-fig})--(\subref{526273a-fig}), we use the same approach.  
It is similar to what we did in (1), but also reuses aspects of our approach in
the proof of Lemma~\ref{main-reduc-lem}. 

(\subref{526271a-fig}) Let $G':=G-\{v_1,v_2\}$ and $G'':=G'-\{w_1,w_2,w_3\}$.  Define $\a',\b',\a'',\b''$ analogously to above.  
By minimality, we have a $k$-good recoloring sequence $\sigma''$ for $G''$ from $\a''$ to $\b''$.  
We extend $\sigma''$ simultaneously to $\{w_1,w_2,w_3\}$.  Most of the time, none of $w_1,w_2,w_3$ defers to each other.
But for the first 40 colors that appear in $N_G(w_1)\setminus \{w_2\}$, we have vertex $w_2$ defer to $w_1$. 
So the number of times $w_1$ is recolored is at most $(4k-40)/4+40/5+2 = k$.  Similarly, $w_2$ defers to $w_3$ for the first 
40 colors that appear in $N_G(w_3)\setminus\{w_2\}$.  (So the analysis for $w_3$ is identical to that for $w_1$.)
Thus, the number of times that $w_2$ is recolored is at most $\ceil{3k/4}+40/5+40/5+2\le 3k/4+19\le k-8$, since $k\ge 108$.
Finally, we extend $\sigma'$ to $v_1$ and $v_2$, with neither defering to the other.
So the number of times $v_1$ is recolored is at most $(4k-8)/4+2 = k$, as desired.  
For $v_2$, the argument is identical.

\begin{figure}[!h]
\ContinuedFloat
\centering
\setcounter{subfigure}{30}
\begin{subfigure}[b]{0.10\linewidth}
\centering
\begin{tikzpicture}[scale=0.8, yscale=.866, xscale=.5]
\tikzset{every node/.style=6vertB}
\begin{scope}[yscale=-1]
\draw 
(0,0) node[5vertB, inner sep=.5pt] (v1) {\footnotesize{$v_1$}}
--++ (1,1) node[6vertB] (w) {\footnotesize{$w$}}
--++ (1,-1) node[5vertB, inner sep=.5pt] (v2) {\footnotesize{$v_2$}} -- (v1);
\end{scope}
\end{tikzpicture}
\subcaption{\label{5261a-fig}}
\end{subfigure}
\begin{subfigure}[b]{.17\linewidth}
\begin{tikzpicture}[scale=1.0, yscale=.866, xscale=.5]
\begin{scope}
\draw 
(0,0) node[5vertB, inner sep=.5pt] (v1) {\footnotesize{$v_2$}}
++ (2,0) node[5vertB, inner sep=.5pt] (v2) {\footnotesize{$v_3$}}
--++ (1,-1) node[5vertB, inner sep=.5pt] (w3) {\footnotesize{$v_4$}} -- (v2)
--++ (-1,-1) node[8vertB] (w4) {\footnotesize{$w$}} -- (v1)
--++ (-1,-1) node[5vertB, inner sep=.5pt] (w5) {\footnotesize{$v_1$}};
\draw (w3) -- (w4) -- (w5);
\end{scope}
\end{tikzpicture}
\subcaption{\label{5481a-fig}~~~}
\end{subfigure}
\begin{subfigure}[b]{.17\linewidth}
\begin{tikzpicture}[scale=1.0, yscale=.866, xscale=.5]
\begin{scope}
\draw 
(0,0) node[5vertB, inner sep=.5pt] (v1) {\footnotesize{$v_1$}}
--++ (2,0) node[5vertB, inner sep=.5pt] (v2) {\footnotesize{$v_2$}}
--++ (1,-1) node[6vertB] (w3) {\footnotesize{$w_3$}} -- (v2)
--++ (-1,-1) node[7vertB, inner sep=0pt] (w4) {\footnotesize{$w_2$}} -- (v1)
--++ (-1,-1) node[6vertB] (w5) {\footnotesize{$w_1$}};
\draw (w3) -- (w4) -- (w5);
\end{scope}
\end{tikzpicture}
\subcaption{\label{526271a-fig}~~}
\end{subfigure}
\begin{subfigure}[b]{.17\linewidth}
\begin{tikzpicture}[scale=1.0, yscale=.866, xscale=.5]
\begin{scope}
\draw 
(0,0) node[5vertB, inner sep=.5pt] (v1) {\footnotesize{$v_1$}}
--++ (2,0) node[5vertB, inner sep=.5pt] (v2) {\footnotesize{$v_2$}}
  ++ (1,-1) node[5vertB, inner sep=.5pt] (w3) {\footnotesize{$v_4$}} (v2)
--++ (-1,-1) node[7vertB] (w4) {\footnotesize{$w$}} -- (v1)
  ++ (-1,-1) node[5vertB, inner sep=.5pt] (w5) {\footnotesize{$v_3$}};
\draw (w3) -- (w4) -- (w5);
\end{scope}
\end{tikzpicture}
\subcaption{\label{5471a-fig}~~~}
\end{subfigure}
\begin{subfigure}[b]{.17\linewidth}
\begin{tikzpicture}[scale=1.0, yscale=.866, xscale=.5]
\begin{scope}
\draw 
(0,0) node[5vertB, inner sep=.5pt] (v1) {\footnotesize{$v_1$}}
--++ (2,0) node[5vertB, inner sep=.5pt] (v2) {\footnotesize{$v_2$}}
--++ (1,1) node (w2) {\footnotesize{$w_3$}} 
--++ (0,-2) node (w3) {\footnotesize{$w_2$}} -- (v2)
--++ (-1,-1) node[7vertB, inner sep=0pt] (w4) {\footnotesize{$w_1$}} -- (v1) -- (v2);
\draw (w3) -- (w4);
\end{scope}
\end{tikzpicture}
\subcaption{\label{526271b-fig}~~~~~~~~~~~~~}
\end{subfigure}
\begin{subfigure}[b]{.17\linewidth}
\begin{tikzpicture}[scale=1.0, yscale=.866, xscale=.5]
\begin{scope}
\draw 
(0,0) node[5vertB, inner sep=.5pt] (v1) {\footnotesize{$v_1$}}
--++ (1,1) node[7vertB, inner sep=0pt] (w1) {\footnotesize{$w_4$}}
--++ (1,-1) node[5vertB, inner sep=.5pt] (v2) {\footnotesize{$v_2$}}
--++ (1,1) node (w2) {\footnotesize{$w_3$}} 
--++ (0,-2) node[7vertB, inner sep=0pt] (w3) {\tiny{~}\footnotesize{$x$}\tiny{~~}} -- (v2)
--++ (-1,-1) node[7vertB, inner sep=0pt] (w4) {\footnotesize{$w_2$}} -- (v1)
--++ (-1,-1) node (w5) {\footnotesize{$w_1$}};
\draw (w1) -- (w2) (w3) -- (w4) -- (w5) (v1) -- (v2);
\end{scope}
\end{tikzpicture}
\subcaption{\label{526273a-fig}~~}
\end{subfigure}
\stepcounter{figure}
\caption{6 more reducible configurations\label{more-reduc-fig}}
\end{figure}
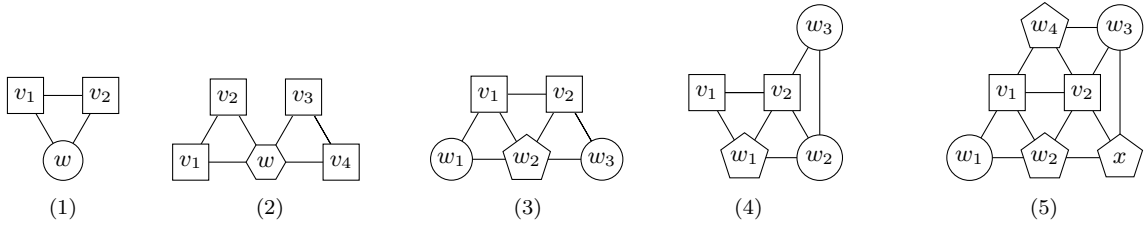
(\subref{5471a-fig}) The proof is nearly identical to that of (\subref{526271a-fig}), but with $(v_3,v_4,w)$ in place of 
$(w_1,w_3,w_2)$.  What were $6$-vertices previously are now $5$-vertices, so they can each afford to have one less neighbor 
in the configuration.

(\subref{526271b-fig}) We just provide a sketch. We get a $k$-good recoloring sequence for $G-\{v_1,v_2,w_1,w_2,w_3\}$.
We extend it simultaneously to $w_1,w_2,w_3$.  We have $w_2$ defer to $w_1$ for the first $200$ colors that appear in 
$N_G(w_1)\setminus \{w_2\}$.  
So the number of times that $w_1$ is recolored is at most $(4k-200)/4+200/5+2 = k-50+40+2 = k-8$.
Similarly, $w_2$ defers to $w_3$ for the first $40$ colors that appear in $N_G(w_3)\setminus\{w_2\}$.
So the number of times that $w_3$ is recolored is at most $(4k-40)/4+40/5+2 = k-10+8+2 = k$.
And the number of times that $w_2$ is recolored is at most $\ceil{3k/4}+200/5+40/5+2 \le \ceil{3k/4} + 50\le k$, since $k\ge 200$.
Finally, we extend this recoloring sequence to $v_1$ and $v_2$, with neither defering to the other.
Now the number of times that $v_1$ is recolored is at most $(4k-8)/4+2=k$, as desired.
The analysis for $v_2$ is identical.

(\subref{526273a-fig}) Again, we just provide a sketch. We first get a $k$-good recoloring sequence for $G-\{v_1,v_2,w_1,w_2,w_3,w_4,x\}$.
By the Extension Lemma, we extend it to $x$.  So the number of times $x$ gets recolored is at most $\ceil{4k/5}+2\le k-72$, since $k\ge 5(74)=370$.  Now we simultaneously extend to $w_1$ and $w_2$, except that $w_2$ defers to $w_1$ for the first $40$ colors 
appearing in $N_G(w_1)\setminus\{w_2\}$.  So the number of times $w_2$ is recolored is at most $(4k-72)/4 + 40/5 + 2 =k-8$.
And the number of times $w_1$ is recolored is at most $(4k-40)/4+40/5+2 = k$.  When we simultaneously extend to $w_3$ and $w_4$,
the analysis is identical, so we omit it.  Finally, we extend to $v_1$ and $v_2$.  The number of times $v_1$
is recolored is at most $(4k-8)/4+2 = k$.  And the analysis for $v_2$ is identical.
\end{proof}

\begin{lem}
\label{edge-remove-lem}
Let $H$ be any reducible configuration from Lemmas~\ref{first-reduc-lem}, \ref{second-reduc-lem}, \ref{main-reduc-lem}, and 
\ref{more-reduc-lem}, shown in Figures~\ref{main-reduc-fig} and~\ref{more-reduc-fig}.
If we delete one or more edges of $H$, and decrease the degree of each endpoint for each deleted incident edge, then the result
also contains a reducible configuration.
\end{lem}

\begin{proof}
We start with Lemma~\ref{more-reduc-lem}.  In (RC\subref{5261a-fig}), (RC\subref{5481a-fig}), and (RC\subref{5471a-fig}), 
if we delete any edge, then we get a $4^-$-vertex, which is reducible by the Extension Lemma.
In (RC\subref{526271a-fig}), (RC\subref{526271b-fig}), and (RC\subref{526273a-fig}), this is also true for many edges; 
but for the others, we get a copy of either (RC\subref{53a-fig}) or (RC\subref{5261a-fig}).

Now we consider Lemma~\ref{second-reduc-lem}.  Each time we delete an edge $e$, each endpoint of $e$ other than $v$ becomes a 
$4^-$-vertex, and we are done, or it becomes a $5_2$-vertex (as in the proof of Lemma~\ref{main-reduc-lem}).  In the latter case it 
is a well and can support itself; all other vertices continue to be supported by $v$, as in the proof of Lemma~\ref{second-reduc-lem}.

Finally, we consider Lemma~\ref{main-reduc-lem} (and Lemma~\ref{first-reduc-lem}).  All vertices are $5_{1^+}$-, 
$6_{3^+}$-, or $7_{5^+}$-vertices.  If we delete an edge incident to a $5$-vertex, it becomes a $4^-$-vertex, so we are done.  
But if we delete an edge incident to a $6_{3^+}$-vertex, it becomes a $5_{2^+}$-vertex; and a $7_{5^+}$-vertex becomes a 
$6_{4^+}$-vertex.  Thus, each endpoint of a deleted edge $e$ becomes a well.  As such, it can be the root of the component of 
$T-e$ containing it, where $T$ was the tree in $\prune(\Jc)$ previously containing it.
\end{proof}

\begin{lem}
All of the configurations in Lemmas~\ref{first-reduc-lem}--\ref{more-reduc-lem} are $k$-reducible for all $k\ge 520$.
\end{lem}
\begin{proof}
As we mentioned above, the particular value of $k$ is not all that important.  However, with a bit more work, we can improve the
hypothesis $k\ge 828$ to $k\ge 520$.  That is what we do now.

First note that Lemmas~\ref{second-reduc-lem} and~\ref{more-reduc-lem} already work for $k\ge 520$.  And Lemma~\ref{first-reduc-lem}
is a special case of Lemma~\ref{main-reduc-lem}.  So we only need to improve Lemma~\ref{main-reduc-lem}.  Of the 4 out-trees in
its proof, all but out-tree (2) work for $k\ge 520$.  So in what follows, we consider out-tree (2).

When we handled (2) in the proof of Lemma~\ref{main-reduc-lem}, we were concerned about having recolored $x$ to avoid colors
$c_{i}$, $c_{i+1}$, and $c_{i+2}$, with $w$ defering to $x$, but then needing to recolor $x$ (because $x$ defers to $y$) before
color $c_{i}$ appears on $N_G(x)\setminus N_H(x)$.  At this point, $x$ will only be able to avoid colors $c_{i}$ and $c_{i+1}$,
and will ``lose ground'' of one color in this sequence of colors.

Now we are more selective about when we make $w$ defer to $x$, so that $x$ can avoid colors $c_{i}$, $c_{i+1}$, $c_{i+2}$.  
For this to happen, we require that $w$ has defered to $x$ fewer than
$30$ times (and that $x$ needs to be recolored because its color is about to appear on $N_G(x)\setminus N_H(x)$).
But now we also require that it is not the case that both (a) the next color $d_j$ appears on $N_G(y)\setminus N_H(y)$ before the 
next color $c_i$ appears on $N_G(x)\setminus N_H(x)$ and also (b) $j\le 40$.  So there are at most $40$ times that $w$ does not
defer to $x$ before $w$ defers $30$ times.
In this way, $x$ will never lose ground, since either $j>40$ and $x$ no longer must defer to $y$, or else color $c_i$ will appear
on $N_G(x)\setminus N_H(x)$ before $x$ defers to $y$, so $x$ will only need to avoid $c_{i+1}$ and $c_{i+2}$, when we recolor $x$
to defer to $y$.

We also do something similar for $v$ defering to $w$.  For this to happen, we require that $v$ has defered to $w$ fewer than
$96$ times (and that $w$ needs to be recolored because its color is about to appear on $N_G(w)\setminus N_H(w)$).
But now we also require that it is not the case that both (a) the next color $c_j$ appears on $N_G(x)\setminus N_H(x)$ before the 
next color $b_i$ appears on $N_G(w)\setminus N_H(w)$ and also (b) $j\le 160$.  So there are at most $160$ times that $v$ does not
defer to $w$ before $v$ defers $96$ times.

Now the computations for out-tree (2) are as follows.  The number of times that $y$ is recolored is $(4k-40)/4+40/5+2=k$.
The number of times that $x$ is recolored is at most $(3k-120)/3+120/4+40/5+2=k$.  The number of times that $w$ is recolored
is at most $(3k-384)/3+384/4+120/4+2=k$.  The number of times that $z$ is recolored is at most $(3k-24)/3+24/4+2\le k$.
Finally, the number of times that $v$ is recolored is at most $\lceil(3k+96+6)/4\rceil+96+6+2\le k$, since $k\ge 520$.
\end{proof}

\section{Discharging}
\label{unavoidability-sec}
\label{discharging-sec}
In this section, we prove the Main Theorem.  To do so, we assume it is false, and let $G$ be a $k$-minimal counterexample 
(with $k=520$, say).  We use discharging to prove that $G$ must contain one of the configurations from the previous 
section, all of which we showed cannot appear in a $k$-minimal counterexample.  This contradiction proves the Main Theorem.
Before describing our discharging argument, we make the following observation.

\begin{obs}
Our graph $G$ has minimum degree at least $5$.
\end{obs}
\begin{proof}
This follows directly from the Extension Lemma, analogous to the proof of Lemma~\ref{linear-cereceda:lem}.
\end{proof}

To simplify the discharging arguments, if $G$ is not a triangulation, then we consider a supergraph $G^+$ of $G$, with $|G^+|=|G|$,
that is a plane triangulation.  In this section, we prove that $G^+$ contains as a configuration $H$ either a $4^-$-vertex or one of
(RC\subref{53a-fig})--(RC\subref{526273a-fig}).  By Lemma~\ref{edge-remove-lem}, when we restrict $H$ to its edges in $G$, we also get a reducible
configuration for $G$.

We assign to each vertex $v$ in $G$ a ``charge'' $d(v)-6$.  By Euler's formula, $\sum_{v\in V(G)}(d(v)-6) = 2|E(G)|-6|G| = 2(3|G|-6)-6|G| = -12$.  
Assuming that $G$ contains none of the configurations in the previous section, we 
redistribute charge so that each vertex $v$ ends with nonnegative charge; we say that such a vertex $v$ \Emph{ends happy}.
This yields an obvious contradiction, since the sum of nonnegative quantities is equal to $-12$.  (And this contradiction proves
our Main Theorem.) To redistribute charge, we use the following $6$ discharging rules, applied successively.
An $(h,i,j)$-face is a 3-face for which the multiset of degrees of incident vertices is $\{h,i,j\}$.  An $(h^+,i,j)$-face is
defined analogously, with one degree at least $h$.
\begin{itemize}
    \item[\linkdest{R1}{(R1)}] Let $v$ be a $7$-vertex and $w$ be a $5$-neighbor of $v$. If $vw$ lies on a $(7,5,5)$-face, then $v$ sends $1/4$ to $w$;
otherwise, $v$ sends $1/3$ to $w$.
    \item[\linkdest{R2}{(R2)}] If $v$ is a $7$-vertex and $w$ is a $6$-neighbor of $v$ and $vw$ does not lie on a $(7,6,5)$-face, then $v$ sends $1/6$ to $w$.
    \item[\linkdest{R3}{(R3)}] If $v$ is an $8^+$-vertex and $w$ is a $5$-neighbor of $v$, then $v$ sends $1/2$ to $w$.
    \item[\linkdest{R4}{(R4)}] If $v$ is an $8^+$-vertex and $w$ is a $6$-neighbor of $v$ and $vw$ does not lie on an $(8^+,6,5)$-face, then $v$ sends $1/4$ to $w$.
    \item[\linkdest{R5}{(R5)}]  
If a $6$-vertex $v$ has charge, but has no $5$-neighbor, then $v$ splits its charge equally among its $6$-neighbors (if any) that do have $5$-neighbors.
    \item[\linkdest{R6}{(R6)}] \label{R6} Each $6$-vertex with a $5$-neighbor splits its charge equally among its $5$-neighbors.
\end{itemize}

In the rest of the section, we show that every vertex ends happy (finishes with nonnegative charge).  We first handle the 
$6^+$-vertices, which requires only a straightforward case analysis, based on \Rule{1}--\Rule{4}.  The bulk of our work in this 
section goes to handling $5$-vertices.  For this we consider a number of cases based on the degrees of the neighbors of the 
$5$-vertex.  Since $G$ has no copy of (RC\subref{5165a-fig}), every $5$-vertex has at least one $7^+$-neighbor.
Note that a vertex receiving more charge than claimed never hinders that vertex from ending happy.  Thus, we focus on proving lower
bounds on the charge that each $5$-vertex receives.  With this in mind, we typically write that ``$v$ gives $\mu$ to $w$'' to mean
that vertex $v$ gives vertex $w$ charge at least $\mu$.

\begin{lem}
Each $6^+$-vertex ends happy.
\end{lem}
\label{lem1}
\begin{proof}
Each $6$-vertex starts with $0$.  By \Rule{5} and \Rule{6}, each $6$-vertex also ends with at least $0$ 
(and gives each $5$-neighbor nonnegative charge).

Let $v$ be a $7$-vertex.  Let $s$ be the number of $5$-neighbors of $v$; 
we consider the following cases for $s$.  Since $G$ has no copy of (RC\subref{5471a-fig}), we assume $s\le 3$.
First, $s=0$: If $v$ has seven $6$-neighbors, then $G$ has (RC\subref{6771a-fig}); so $v$ has at most six $6$-neighbors.  Thus, $v$ ends with at least $(7-6)-6(1/6)=0$.  
Next, $s=1$: Now two $6^+$-neighbors of $v$ each have a common (successive) $5$-neighbor with $v$, so cannot get charge from $v$.  Thus, $v$ ends with at least $1-1/3-(7-1-2)(1/6) = 1-1/3-4(1/6)=0$.

Next, $s=2$: Suppose $v$ has successive $5$-neighbors.  They each get $1/4$ from $v$ by \Rule{1}.  Also, two $6^+$-neighbors of $v$ 
have a common (successive) $5$-neighbor with $v$, so do not receive charge from $v$.  Thus, $v$ ends with at least 
$1-2(1/4)-(7-2-2)(1/6)=1-2(1/4)-3(1/6)=0$.  Suppose instead that the $5$-neighbors of $v$ are not successive.  Now each gets $1/3$, 
and at least three $6^+$-neighbors of $v$ get no charge from $v$ by \Rule{2} because they have a common (successive)
$5$-neighbor with $v$.  So $v$ ends with at least $1-2(1/3)-(7-2-3)(1/6)=0$.

Finally, $s=3$: If the three $5$-neighbors of $v$ are pairwise non-successive, then they each get $1/3$, and no other neighbor of 
$v$ gets charge by \Rule{2}.  So $v$ ends with $1-3(1/3)=0$.  Suppose instead that two $5$-neighbors of $v$ are succesive.
Each gets $1/4$ from $v$.  (Each $5$-vertex has at most one $5$-neighbor, or $G$ contains (RC\subref{53a-fig}).)  
The other $5$-neighbor of $v$ gets $1/3$.  Finally, at most one $6^+$-neighbor of $v$ gets charge $1/6$ from $v$.  
So $v$ ends with at least $1-2(1/4)-1/3-1/6=0$.
\smallskip

Now let $v$ be an $8$-vertex.  Again, let $s$ be the number of $5$-neighbors of $v$.
If $s=0$, then $v$ ends with at least $8-6-8(1/4) =0$.  If $s=1$, then $v$ ends with at least $2-1/2-(8-2-1)(1/4) = 2-1/2-5(1/4) > 0$. If $s=2$, then $v$ ends with at least $2-2(1/2)-4(1/4)=0$. If $s=3$, then not all three of its $5$-neighbors are consecutive, or
$G$ contains (RC\subref{53a-fig}), so $v$ ends with at least $2-3(1/2)-2(1/4)=0$.
Finally, if $s=4$, then $v$ cannot have three successive $5$-neighbors, or $G$ contains (RC\subref{53a-fig}); and $v$ also cannot have 
two pairs of successive $5$-neighbors, or $G$ contains (RC\subref{5481a-fig}).  
So each $6^+$-neighbor of $v$ has a successive common $5$-neighbor with $v$.  Thus, no $6^+$-neighbors
of $v$ get charge by \Rule{4}.  Hence, $v$ ends with $2-4(1/2)=0$.

Finally, let $v$ be a $9^+$-vertex.  If we consider any three successive neighbors of $v$, then at most two of them are 
$5$-neighbors, or $G$ contains (RC\subref{53a-fig}).  And if two of these are $5$-neighbors, then the third gets no charge 
from $v$ by \Rule{4}.  Thus, the average charge that $v$ gives to any three consecutive neighbors is at most $(1/2+1/2+0)/3=1/3$.  
As a result, the average charge that $v$ gives to all of its
neighbors is at most $1/3$.  So $v$ ends with at least $d(v)-6-d(v)/3 = 2(d(v)-9)/3 \ge 0$.
\end{proof}

\begin{lem}
Each pair of adjacent $5$-vertices ends happy.
\label{lem2}
\label{5-5lem}
\end{lem}
\noindent
\textit{Proof.}
Let $v_1$ and $v_2$ be adjacent $5$-vertices and label the vertices of $N[\{v_1,v_2\}]$ as in Figure~\ref{5-5fig}.
Recall from (RC\subref{53a-fig}) that each component induced by $5$-vertices has at most 2 vertices.
So each $w_i$ is a $6^+$-vertex.
Note that $w_1$ must be a $7^+$-vertex; otherwise, $G[\{v_1,v_2,w_1\}]$ is (RC\subref{5261a-fig}).  
Similarly, $w_4$ is a $7^+$-vertex.

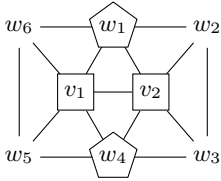
\begin{wrapfigure}[9]{l}{0.35\textwidth}
\centering
\begin{tikzpicture}[scale=1.0, yscale=.866, xscale=.5]
\tikzset{every node/.style=6vertB}
\draw 
(0,0) node[5vertB, inner sep=.5pt] (v1) {\footnotesize{$v_1$}}
--++ (1,1) node[7vertB, inner sep=0pt] (w1) {\footnotesize{$w_1$}}
--++ (1,-1) node[5vertB, inner sep=.5pt] (v2) {\footnotesize{$v_2$}}
--++ (1.5,1) node[unvertB] (w2) {\footnotesize{$w_2$}} 
--++ (0,-2) node[unvertB] (w3) {\footnotesize{$w_3$}} -- (v2)
--++ (-1,-1) node[7vertB, inner sep=0pt] (w4) {\footnotesize{$w_4$}} -- (v1)
--++ (-1.5,-1) node[unvertB] (w5) {\footnotesize{$w_5$}} 
--++ (0,2) node[unvertB] (w6) {\footnotesize{$w_6$}} --  (v1) -- (v2);
\draw (w6) -- (w1) -- (w2) (w3) -- (w4) -- (w5);
\end{tikzpicture}
\caption{Adjacent 5-vertices\label{5-5fig}}
\end{wrapfigure}

First suppose $w_1$ is an $8^+$-vertex.  If $w_4$ is also an $8^+$-vertex, 
then $v_1$ and $v_2$ each receive at least $2(1/2)=1$, so we are done.
Suppose instead that $w_4$ is a $7$-vertex.  
Now one of $w_2$ and $w_3$ is a $7^+$-vertex, or $N[v_2]$ contains (RC\subref{526271b-fig}).  
Similarly, one of $w_5$ and $w_6$ is a $7^+$-vertex.  
Now $v_1$ and $v_2$ each receive at least $1/2+1/3+1/4>1$, and we are done.  

Instead assume $w_1$ and $w_4$ are both $7$-vertices.  First suppose $w_2,w_3,w_5,w_6$ are all $7^-$-vertices.  Now we show that 
at least three of $w_2$, $w_3$, $w_5$, $w_6$ are $7$-vertices.  If $w_2$ and $w_3$ are both $6$-vertices, then 
$G[\{v_1,v_2,w_1,w_2,w_3\}]$ is (RC\subref{526271b-fig}).  And if $w_2$ and $w_6$ are both $6$-vertices, 
then $G[\{v_1,v_2,w_1,w_2,w_6\}]$ is (RC\subref{526271a-fig}).  
So at least one vertex in $\{w_2,w_3\}$ is a $7$-vertex.  Similarly for $\{w_3,w_5\}$; and $\{w_5,w_6\}$; 
and $\{w_6,w_2\}$.  If at most two of $w_2,w_3,w_5,w_6$ are $7$-vertices, then by symmetry we assume $w_2,w_5$ are 
$6$-vertices and $w_3,w_6$ are $7$-vertices.  But now $G[\{v_1,v_2,w_1,w_2,w_3,w_4,w_5\}]$ is (RC\subref{526273a-fig}).
Thus, at least three of $w_2$, $w_3$, $w_5$, and $w_6$ are $7$-vertices.  So the total received by $v_1$ and $v_2$
is at least $4(1/4)+3(1/3)=2$; thus, $v_1$ and $v_2$ end happy.

Assume instead 
that $w_2$ is an $8^+$-vertex.  
If any of $w_3,w_5,w_6$ is an $8^+$-vertex, then in total $v_1$ and $v_2$ get $2(1/2)+4(1/4)=2$.  So assume $w_3,w_5,w_6$ 
are $7^-$-vertices.  If at least two of them are $7$-vertices, then in total $v_1$ and $v_2$ get $1/2+4(1/4)+2(1/3)>2$.
So assume two of $w_3,w_5,w_6$ are $6$-vertices and one is a $7$-vertex.  
If $w_3$ and $w_5$ are $6$-vertices, then $G[\{v_1,v_2,w_3,w_4,w_5\}]$ is (RC\subref{526271a-fig}).
If $w_5$ and $w_6$ are $6$-vertices, then 
$G[\{v_1,v_2,w_4,w_5,w_6\}]$ is (RC\subref{526271b-fig}).
So assume $w_3$ and $w_6$ are $6$-vertices, and $w_5$ is a $7$-vertex.  
Now $G[\{v_1,v_2,w_1,w_3,w_4,w_5,w_6\}]$ is (RC\subref{526273a-fig}).
\hfill$\qed$

\begin{lem}
Each $5$-vertex with three (or more) $7^+$-neighbors ends happy.
\end{lem}
\begin{proof}
By Lemma~\ref{lem2}, we assume $v$ has no $5$-neighbor.  So $v$ ends with at least $-1+3(1/3)=0$.
\end{proof}

Before we start handling various possibilities for the neighborhood of a $5$-vertex $v$, we first prove a few lemmas about
how much charge reaches $v$ from vertices at distance $2$.

\begin{lem}
Let $v$ be a $5$-vertex, and let $w_2,w_3,x_3$ be $6$-vertices, such that $v,w_2,w_3,x_3$ have all edges among them shown in Figure~\ref{helper1-fig}.  
If $w_2$ and $w_3$ each have no $5$-neighbor but $v$, then $x_3$ sends $v$ at least $1/6$ via $w_2$ and $w_3$. 
\label{helper1-lem}
\end{lem}

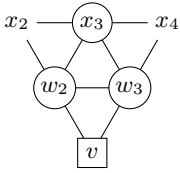
\begin{wrapfigure}[9]{l}{0.3\textwidth}
\centering
\begin{tikzpicture}[rotate = 60, yscale=.866, xscale=-.5]

\tikzset{every node/.style=6vertB}

\draw 
(1,1) node (x3) {\footnotesize{$x_3$}}
--++ (1,-1) node (w3) {\footnotesize{$w_3$}}
--++ (1,1) node (w2) {\footnotesize{$w_2$}}
--++ (1,-1) node[5vertB] (v) {\footnotesize{$v$}} 
(x3) --++ (1,1) node[unvertB] (x2) {\footnotesize{$x_2$}} -- (w2)
(x3) --++ (-1,-1) node[unvertB] (x4) {\footnotesize{$x_4$}} -- (w3)
(w3) -- (v) (x3) -- (w2);
\end{tikzpicture}
\captionsetup{width=.225\textwidth}
\caption{Vertex $x_3$ gives $1/6$ to $v$ via $w_2$ and $w_3$\label{helper1-fig}}
\end{wrapfigure}

\noindent
\textit{Proof.}
First note that $x_3$ has no $5$-neighbor $y$, since then $G[\{v, w_2, w_3, x_3,y\}]$ contains (RC\subref{5263a-fig}).
Note also that $x_3$ must have at least one $7^+$-neighbor; otherwise, $G[N[x_3]]$ contains (RC\subref{67a-fig}).
If $x_3$ has at least two $7^+$-neighbors, then they each give $1/6$ to $x_3$ (since $x_3$ has no $5$-neighbors).  Now $x_3$ 
has at least $1/3$ and splits it among at most four vertices by \Rule{5}, two of which are $w_2$ and $w_3$.  
Since $w_2$ and $w_3$ each have no $5$-neighbor but $v$, they send all of this charge to $v$ by \Rule{6}.  So $v$ gets $1/6$, as desired.

So assume that $x_3$ has only a single $7^+$-neighbor; note that it gives $1/6$ to $x_3$ by (R2) or (R3).  We will show that $w_2$
and $w_3$ are the only $6$-neighbors of $x_3$ with a 5-neighbor, so all the charge that $x_3$ gets continues on to $w_2$ and $w_3$
by \Rule{5}, and eventually to $v$ by \Rule{6}.  Suppose that $x_2$ (or $x_4$) is a $7^+$-vertex and all other neighbors of $x_3$ 
are $6$-vertices.  If any neighbor of $x_3$ other than $w_2$ or $w_3$ has a $5$-neighbor, then $G$ contains 
(RC\subref{5264a-fig}), (RC\subref{5265a-fig}), or (RC\subref{5266b-fig}).  
If not, then $x_3$ gets $1/6$ from $x_2$ by \Rule{2}, and sends it all to $v$ via $w_2$ and $w_3$, by \Rule{6}.

Finally, suppose $x_3$ has a single $7^+$-neighbor, a vertex other than $x_2$ and $x_4$; call it $y$.  
If $x_2$ has a $5$-neighbor $y'$, then $G[\{v,w_2,w_3,x_2,x_3,y'\}]$ 
contains (RC\subref{5264a-fig}).  So $x_2$ has no $5$-neighbor; likewise for $x_4$.
A similar argument works for the neighbors of $x_3$ outside of $\{w_2,w_3,x_2,x_4\}$, or else $G$ contains (RC\subref{5265a-fig}).
So $x_3$ receives $1/6$ from its $7^+$-neighbor and passes it all to $v$ via $w_2$ and $w_3$. 
\hfill$\qed$

\bigskip

A $6$-vertex is \Emph{needy} if it has at least one $5$-neighbor. 

\begin{rem}
    Occasionally, we mention a figure slightly before we have proved all the degree bounds that it illustrates.  This is intended
    to offer the reader support as early as possible; and typically, we prove the remaining degree bounds shortly thereafter.
    Similarly, our figures often illustrate subcases (where more degree bounds are known), rather than the general cases.
\end{rem}

\begin{lem}
Let $v$ be a $5$-vertex, and $w_2,w_3,x_3$ be $6$-vertices. Suppose $v,w_2,w_3,x_2,x_3$ induce all edges shown among them
in Figure~\ref{helper2-fig}.  If $v$ is $w$'s only $5$-neighbor, then $x_2$ gives $1/6$ to $v$ via $w_2$.
\label{helper2-lem}
\end{lem}

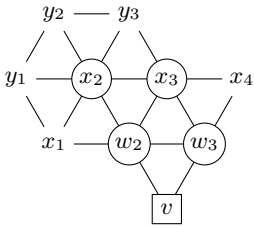
\begin{wrapfigure}[11]{l}{0.3\textwidth}
\centering
\begin{tikzpicture}
\begin{scope}[rotate=60, yscale=.866, xscale=-.5]
\tikzset{every node/.style=6vertB}

\draw 
(1,1) node (x3) {\footnotesize{$x_3$}}
--++ (1,-1) node (w3) {\footnotesize{$w_3$}}
--++ (1,1) node (w2) {\footnotesize{$w_2$}}
--++ (1,-1) node[5vertB] (v) {\footnotesize{$v$}} 
(x3) --++ (1,1) node (x2) {\footnotesize{$x_2$}} -- (w2)
(x3) --++ (-1,-1) node[unvertB] (x4) {\footnotesize{$x_4$}} -- (w3)
(w3) -- (v) (x3) -- (w2);
\draw (4,2) node[unvertB] (x1) {\footnotesize{$x_1$}} --
(3,3) node[unvertB] (y1) {\footnotesize{$y_1$}} --
(1,3) node[unvertB] (y2) {\footnotesize{$y_2$}} --
(0,2) node[unvertB] (y3) {\footnotesize{$y_3$}};
\draw (w2) -- (x1) -- (x2) -- (y1) (y2) -- (x2) -- (y3) -- (x3);
\end{scope}
\end{tikzpicture}
\captionsetup{width=.225\textwidth}
\caption{Vertex $x_2$ gives $1/6$ to $v$ via $w_2$\label{helper2-fig}}
\end{wrapfigure}

\noindent
\textit{Proof.}
The proof is similar to that of Lemma~\ref{helper1-lem} above.  If $x_2$ is a $7^+$-vertex, then it gives $1/6$ to $v$ via $w_2$.
So assume $x_2$ is a $6$-vertex.

None of $x_1,x_3,y_3$ is needy, or $G$ contains (RC\subref{5263a-fig}) or (RC\subref{5265a-fig}).  Similarly, at most one of
$y_1$ and $y_2$ is needy, or $G$ contains~(RC\subref{5366a-fig}).  If at least two of $x_1,y_1,y_2,y_3$ is a $7^+$-vertex,
then $x_2$ gets at least $2(1/6)$ and gives at least half of this to $w_2$, and on to $v$.  Suppose instead that exactly one of
$x_1,y_1,y_2,y_3$ is a $7^+$-vertex.  

Now no neighbors of $x_2$ are needy except for $w_2$, or else $G$ contains (RC\subref{5266a-fig}),
(RC\subref{5267a-fig}), (RC\subref{5267b-fig}), or (RC\subref{526571a-fig}) with a $7$-vertex replaced by a $6$-vertex, which we 
denote by (RC\subref{526571a-fig}$'$).
Specifically, we list the possible triples of (respectively) a $7^+$-neighbor of $x_2$, a needy neighbor of $x_2$, and the 
reducible configuration: 
$(x_1,y_1,$ RC\subref{5267a-fig}), 
$(x_1,y_2,$ RC\subref{5266a-fig}), 
$(y_1,y_2,$ RC\subref{5266a-fig}), 
$(y_2,y_1,$ RC\subref{526571a-fig}$'$), 
$(y_3,y_1,$ RC\subref{526571a-fig}$'$), and
$(y_3,y_2,$ RC\subref{5267b-fig}). 
\hfill$\qed$

\begin{lem}
Each $5$-vertex with a $7^+$-neighbor and another $8^+$-neighbor ends happy.
\end{lem}
\noindent
\textit{Proof.}
If a $5$-vertex $v$ has at least two $8^+$-neighbors, then $v$ gets $2(1/2)=1$, and we are done.
So assume that the $7^+$-neighbor is a $7$-neighbor.  

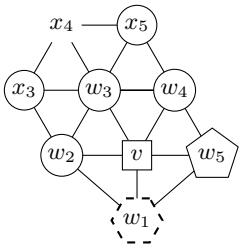
\begin{wrapfigure}[11]{l}{0.3\textwidth}
\centering
\begin{tikzpicture}[rotate=90, rotate=30, yscale=.866, xscale=-.5]

\tikzset{every node/.style=6vertB}

\draw 
(1,1) node (x3) {\footnotesize{$x_3$}}
--++ (1,-1) node (w3) {\footnotesize{$w_3$}}
--++ (1,1) node (w2) {\footnotesize{$w_2$}}
--++ (1,-1) node[5vertB] (v) {\footnotesize{$v$}} 
(x3) --++ (-1,-1) node[unvertB] (x4) {\footnotesize{$x_4$}} -- (w3)
(w3) -- (v) (x3) -- (w2);
\draw (3,-1) node (w4) {\footnotesize{$w_4$}};
\draw (5,-1) node[7vertB] (w5) {\footnotesize{$w_5$}};
\draw[dashed, thick] (5.5,0.5) node[8vertB] (w1) {\footnotesize{$w_1$}};
\draw (w3) -- (w4) -- (w5) -- (w1) 
(w2) -- (w1) -- (v) -- (w5) (w3) -- (w4) -- (v);
\draw (1,-1) node (x5) {\footnotesize{$x_5$}};
\draw (x4) -- (x5) -- (w4) (w3) -- (x5);
\end{tikzpicture}
\captionsetup{width=.275\textwidth}
\caption{A 5-vertex with an $8^+$-neighbor, a $7$-neighbor, and three succesive $6$-neighbors\label{5-87666afig}}
\end{wrapfigure}

Suppose that $w_1$ is an $8^+$-vertex, $w_5$ is a $7$-vertex, and $w_2, w_3, w_4$ are all $6$-vertices;  
see Figure~\ref{5-87666afig}.
Note that $x_3$ is a $6^+$-vertex, or else $G[\{v,w_2,w_3,x_3\}]$ contains (RC\subref{5262a-fig}).
Similarly, $x_5$ is a $6^+$-vertex. 
Suppose that $x_3$ is a $7^+$-vertex.  If $x_4$ is a $6^+$-vertex, then $x_3$ gives $1/6$ to $w_3$ and all of that
charge continues to $v$.  So $v$ gets $1/2+1/3+1/6=1$, and we are done.
But if $x_4$ is a $5$-vertex, then $w_2$ has no $5$-neighbor $y$, or $G[\{v,w_2,w_3,x_4,y\}]$ contains 
(RC\subref{5362a-fig}).  So $x_3$ gives $1/6$ to $w_2$, and this all reaches $v$.
So we instead assume $x_3$ is a $6$-vertex.  Similarly, we assume $x_5$ is a $6$-vertex.  
If $x_4$ is a $6^-$-vertex, then $N[w_3]$ contains (RC\subref{67a-fig}); so assume $x_4$ is a $7^+$-vertex. 
But now $x_4$ gives $1/6$ to $v$ via $w_3$, so we are done.

Suppose instead that $w_1$ is an $8^+$-vertex, that $w_4$ is a $7$-vertex, and that $w_2, w_3, w_5$ are all $6$-vertices;
see Figure~\ref{5-86766fig}.
Note that $x_3$ is a $6^+$-vertex, or else $G[\{v,w_2,w_3,x_3\}]$ contains (RC\subref{5262a-fig}).
If $x_3$ is a $7^+$-vertex, then at most one of $w_2$ and $w_3$ has a $5$-neighbor other than $v$, or $G$ contains
(RC\subref{5362a-fig}).  So $x_3$ gives at least $1/6$ to $w_2$ or $w_3$ that all continues on to $v$; 
now $v$ gets at least $1/2+1/3+1/6=1$. 
Thus, we assume $x_3$ is a $6$-vertex.  

\begin{wrapfigure}[12]{l}{0.3\textwidth}
\centering
\begin{tikzpicture}[rotate=90, rotate=30, yscale=.866, xscale=-.5]
\tikzset{every node/.style=6vertB}

\draw 
(1,1) node (x3) {\footnotesize{$x_3$}}
--++ (1,-1) node (w3) {\footnotesize{$w_3$}}
--++ (1,1) node (w2) {\footnotesize{$w_2$}}
--++ (1,-1) node[5vertB] (v) {\footnotesize{$v$}} 
(x3) --++ (1,1) node[unvertB] (x2) {\footnotesize{$x_2$}} -- (w2)
(x3) --++ (-1,-1) node[unvertB] (x4) {\footnotesize{$x_4$}} -- (w3)
(w3) -- (v) (x3) -- (w2);
\draw (w1) -- (4,2) node[unvertB] (x1) {\footnotesize{$x_1$}} 
(0,2) node[unvertB] (y3) {\footnotesize{$y_3$}};
\draw (w2) -- (x1) -- (x2) -- (y3) -- (x3);
\draw (3,-1) node[7vertB] (w4) {\footnotesize{$w_4$}};
\draw (5,-1) node (w5) {\footnotesize{$w_5$}};
\draw[dashed, thick] (5.5,0.5) node[8vertB] (w1) {\footnotesize{$w_1$}};
\draw (w3) -- (w4) -- (w5) -- (w1) 
(w2) -- (w1) -- (v) -- (w5) (w3) -- (w4) -- (v);
\draw (1,-1) node[unvertB] (x5) {\footnotesize{$x_5$}};
\draw (x4) -- (x5) -- (w4) (w3) -- (x5);
\draw (-1,1) node[unvertB] (y4) {\footnotesize{$y_4$}};
\draw (y3) -- (y4) -- (x4) (y4) -- (x3);
\end{tikzpicture}
\captionsetup{width=.275\textwidth}
\caption{A 5-vertex with an $8^+$-neighbor, a $7$-neighbor, and three non-succesive $6$-neighbors\label{5-86766fig}}
\end{wrapfigure}
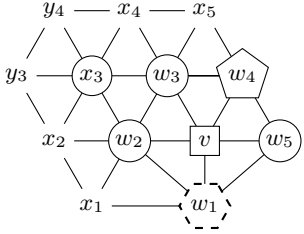

Note that $x_3$ has no $5$-neighbor $y$, or $G[\{v,w_2,w_3,x_3,y\}]$ contains (RC\subref{5263a-fig}).
Note also that one of $x_2$, $x_4$, $y_3$, $y_4$ must be a $7^+$-vertex, or else $G[N[x_3]]$ contains (RC\subref{67a-fig}).
We show that $x_3$ gives at least $1/6$ to $v$ via $w_2$ or $w_3$.  First note that neither $x_2$ nor $x_4$ is needy, or else 
$G$ contains (RC\subref{5264a-fig}).  If both $y_3$ and $y_4$ are needy, say with $5$-neighbors $z'$ and $z''$,
then $G[\{v,w_2,w_3,x_3,y_3,y_4,z',z''\}]$ contains (RC\subref{5365a-fig}).  So at most one of $y_3$ and $y_4$ is needy.

Suppose that at least two of $x_2, x_4, y_3, y_4$ are $7^+$-vertices.  Each of these $7^+$-vertices gives $1/6$ to $x_3$,
so $x_3$ gets $1/3$.  This $1/3$ is split among at most $3$ vertices, two of which are $w_2$ and $w_3$.  As noted above,
at most one of $w_2$ and $w_3$ has a $5$-neighbor other than $v$ (and at most one such $5$-neighbor).  So $1/9$ goes to $w_2$ 
or $w_3$ and on to $v$.  Another $1/9$ goes to the other of $w_2$ and $w_3$, and at least one half of it continues to $v$.
Thus, $v$ receives from $x_3$ (via $w_2$ and $w_3$) at least $1/9+(1/9)/2 = 2/18+1/18 =1/6$.  Hence $v$ gets at least $1/2+1/3+1/6=1$.

Assume instead that exactly one of $x_2, x_4, y_3, y_4$ is a $7^+$-vertex.  If neither $w_2$ nor $w_3$ has a $5$-neighbor (but 
$v$), then $x_3$ gives at least $1/6$ to $v$ (via $w_2$ and $w_3$) by Lemma~\ref{helper1-lem}.  We handle the case that $x_5$ 
is a $5$-vertex, but the argument is similar when $x_1$ is a $5$-vertex.  Now $x_4$ is a $7^+$-vertex, or else $N[w_3]$ contains
(RC\subref{5264a-fig}).  Thus, $x_2$ is a $6$-vertex (since $x_3$ has exactly on $7^+$-neighbor).  
But now $x_2$ has no $5$-neighbor $y$, or else $G[\{v,w_2,w_3,x_3,x_2,y\}]$ contains (RC\subref{5264a-fig}).  
By Lemma~\ref{helper2-lem}, vertex $x_2$ gives $1/6$ to $v$ via $w_2$.
Thus, $v$ gets at least $1/2+1/3+1/6=1$, and we are done.
\hfill$\qed$

\smallskip

\begin{lem}
Each $5$-vertex with a $7^+$-neighbor and four $6$-neighbors ends happy.
\end{lem}

\noindent
\textit{Proof.}
Let $v$ be a $5$-vertex with neighbors $w_1,\ldots,w_5$ in clockwise order.  
Assume that $w_1$ is a $7^+$-vertex and $w_2, w_3, w_4, w_5$ are $6$-vertices.  
We label the remaining vertices as in Figures~\ref{5-76666afig} and~\ref{5-76666bfig}.
Note that $x_3, x_5, x_7$ are all $6^+$-vertices, or $G$ contains (RC\subref{5262a-fig}).
If at least two of $w_2, w_3, w_4, w_5$ have $5$-neighbors other than $v$, then $G$ contains 
(RC\subref{5265a-fig}), (RC\subref{5362a-fig}), or (RC\subref{5363a-fig}).
So, by symmetry, we assume that $w_2$ and $w_3$ both have $v$ as their only $5$-neighbor.

\begin{wrapfigure}[12]{l}{0.3\textwidth}
\centering
~~~~~\\
~~~~~\\
\begin{tikzpicture}[rotate=90, rotate=30, yscale=.866, xscale=-.5]
\tikzset{every node/.style=6vertB}

\draw 
(1,1) node[unvertB] (x3) {\footnotesize{$x_3$}}
--++ (1,-1) node (w3) {\footnotesize{$w_3$}}
--++ (1,1) node (w2) {\footnotesize{$w_2$}}
--++ (1,-1) node[5vertB] (v) {\footnotesize{$v$}} 
(x3) --++ (-1,-1) node[unvertB] (x4) {\footnotesize{$x_4$}} -- (w3)
(w3) -- (v) (x3) -- (w2);
\draw (3,-1) node (w4) {\footnotesize{$w_4$}};
\draw (5,-1) node (w5) {\footnotesize{$w_5$}};
\draw[dashed, thick] (5.5,0.5) node[7vertB] (w1) {\footnotesize{$w_1$}};
\draw (w3) -- (w4) -- (w5) -- (w1) 
(w2) -- (w1) -- (v) -- (w5) (w3) -- (w4) -- (v);
\draw (1,-1) node (x5) {\footnotesize{$x_5$}};
\draw (x4) -- (x5) -- (w4) (w3) -- (x5);
\draw (2,-2) node[unvertB] (x6) {\footnotesize{$x_6$}};
\draw (4,-2) node[unvertB] (x7) {\footnotesize{$x_7$}};
\draw (x5) -- (x6) -- (w4) -- (x7) -- (w5) (x6) -- (x7);
\end{tikzpicture}
\captionsetup{width=.275\textwidth}
\caption{A 5-vertex with a $7^+$-neighbor and four $6$-neighbors; Case 1, when $x_5$ is a $6$-vertex.\label{5-76666afig}}
\end{wrapfigure}

\textbf{Case 1: \bm{$x_5$} is a \bm{$6$}-vertex.} 
Note that $x_4$ is a $6^+$-vertex, or else $G[\{v,w_3,w_4,x_4,x_5\}]$ contains (RC\subref{5263a-fig}).  
Similarly, $x_6$ is a $6^+$-vertex.  
Either $x_6$ or $x_7$ is a $7^+$-vertex, 
or else $G[N[w_4]]$ contains (RC\subref{67a-fig}), with the degree of $v$ decreased from $6$ to $5$;  
so $x_6$ or $x_7$ gives $1/6$ to $v$ via $w_4$. 
By Lemma~\ref{helper1-lem}, vertex $x_5$
gives $1/6$ to $v$ via $w_3$ and $w_4$.  
If $x_3$ is a $7^+$-vertex, then it gives $1/6$ to $v$ via each of $w_2$ and $w_3$.
In this case, we are done, since $v$ gets $1/3+1/6+1/6+1/6+1/6=1$.
Assume instead that $x_3$ is a $6$-vertex.  So $x_4$ is a $7^+$-vertex, or else $G[N[w_3]]$
contains (RC\subref{67a-fig}), again with the degree of $v$ decreased from $6$ to $5$.  
But now $x_4$ gives $1/6$ to $v$ via $w_3$.
And $x_3$ also gives $1/6$ to $v$ via $w_2$ and $w_3$, by Lemma~\ref{helper1-lem}.
Now we are done, since $v$ gets $1/3+1/6+1/6+1/6+1/6=1$.

\begin{wrapfigure}[13]{l}{0.3\textwidth}
\centering
\begin{tikzpicture}[rotate=90, rotate=30, yscale=.866, xscale=-.5]
\tikzset{every node/.style=6vertB}

\draw 
(1,1) node[unvertB] (x3) {\footnotesize{$x_3$}}
--++ (1,-1) node (w3) {\footnotesize{$w_3$}}
--++ (1,1) node (w2) {\footnotesize{$w_2$}}
--++ (1,-1) node[5vertB] (v) {\footnotesize{$v$}} 
(x3) --++ (1,1) node[unvertB] (x2) {\footnotesize{$x_2$}} -- (w2)
(x3) --++ (-1,-1) node[unvertB] (x4) {\footnotesize{$x_4$}} -- (w3)
(w3) -- (v) (x3) -- (w2);
\draw (3,-1) node (w4) {\footnotesize{$w_4$}};
\draw (5,-1) node (w5) {\footnotesize{$w_5$}};
\draw[dashed, thick] (5.5,0.5) node[7vertB] (w1) {\footnotesize{$w_1$}};
\draw (w3) -- (w4) -- (w5) -- (w1) 
(w2) -- (w1) -- (v) -- (w5) (w3) -- (w4) -- (v);
\draw[dashed, thick] (1,-1) node[7vertB] (x5) {\footnotesize{$x_5$}};
\draw (x4) -- (x5) -- (w4) (w3) -- (x5);
\draw (2,-2) node[unvertB] (x6) {\footnotesize{$x_6$}};
\draw (4,-2) node[unvertB] (x7) {\footnotesize{$x_7$}};
\draw (x5) -- (x6) -- (w4) -- (x7) -- (w5) (x6) -- (x7);
\end{tikzpicture}
\captionsetup{width=.275\textwidth}
\caption{A 5-vertex with a $7^+$-neighbor and four $6$-neighbors; Case 2, when $x_5$ is a $7^+$-vertex.\label{5-76666bfig}}
\end{wrapfigure}

\textbf{Case 2: \bm{$x_5$} is a \bm{$7^+$}-vertex.} 
See Figure~\ref{5-76666bfig}.
We first show that $v$ gets a total of $1/3$ from $x_5$ and $x_7$ via $w_3$, 
$w_4$, and $w_5$.  If $x_6$ is a $6^+$-vertex, then $x_5$ gives $1/6$ to $v$ via each of $w_3$ and $w_4$, as desired.
So assume instead that $x_6$ is a $5$-vertex.  Now $x_7$ must be a $7^+$-vertex, by (RC\subref{5263a-fig}).  If any neighbor $x_8$ of $w_5$ (but $v$) is a $5$-vertex, then $G[\{v,w_4,w_5,x_6,x_8\}]$ contains (RC\subref{5362a-fig}); so assume each such $x_8$ is 
a $6^+$-vertex.  But now $x_5$ gives $1/6$ to $v$ 
via $w_3$ and also $x_7$ gives $1/6$ to $v$ via $w_5$, as desired.

We now show that $v$ gets a total of $1/3$ from $x_2$ and $x_3$.  If $x_3$ is a $7^+$-vertex, then it gives $1/6$ to $v$ via
each of $w_2$ and $w_3$.  So assume $x_2$ is a $6$-vertex.  By Lemma~\ref{helper1-lem}, vertex $x_3$ gives $1/6$ to $v$ via $w_2$ 
and $w_3$.  By Lemma~\ref{helper2-lem}, vertex $x_2$ gives $1/6$ to $v$ via $w_2$.  So $v$ gets $1/3+1/3+1/3=1$, and we are done.
\hfill$\qed$

\begin{lem}
Each $5$-vertex with two adjacent $7$-neighbors ends happy.
\end{lem}
\noindent
\textit{Proof.}
Assume that $v$ is a $5$-vertex with neighbors $w_1,\ldots,w_5$ in order, and that $w_1$ and $w_5$ are $7$-vertices.
By Lemma~\ref{5-5lem}, we assume that $v$ has no $5$-neighbors.  If $v$ has another $7^+$-neighbor, then $v$ gets $1/3$ 
from each and
we are done.  So we assume $w_2, w_3, w_4$ are $6$-vertices.  We label the remaining vertices as in Figure~\ref{5-77666fig}.  
Both $x_3$ and $x_5$ are $6^+$-vertices, or else $G$ contains (RC\subref{5262a-fig}).

\begin{wrapfigure}[12]{l}{0.3\textwidth}
\centering
\begin{tikzpicture}[rotate=90, rotate=30, yscale=.866, xscale=-.5]
\tikzset{every node/.style=6vertB}

\draw[thick, dashed] 
(1,1) node (x3) {\footnotesize{$x_3$}};
\draw (x3) --++ (1,-1) node (w3) {\footnotesize{$w_3$}}
--++ (1,1) node (w2) {\footnotesize{$w_2$}}
--++ (1,-1) node[5vertB] (v) {\footnotesize{$v$}} 
(x3) --++ (-1,-1) node[unvertB] (x4) {\footnotesize{$x_4$}} -- (w3)
(w3) -- (v) (x3) -- (w2);
\draw (3,-1) node (w4) {\footnotesize{$w_4$}};
\draw (5,-1) node[7vertB] (w5) {\footnotesize{$w_5$}};
\draw (5.5,0.5) node[7vertB] (w1) {\footnotesize{$w_1$}};
\draw (w3) -- (w4) -- (w5) -- (w1) (w2) -- (w1) -- (v) -- (w5) (w3) -- (w4) -- (v);
\draw[thick, dashed] (1,-1) node (x5) {\footnotesize{$x_5$}};
\draw (x4) -- (x5) -- (w4) (w3) -- (x5);
\end{tikzpicture}

\captionsetup{width=.27\textwidth}
\caption{A 5-vertex with two adjacent $7$-neighbors and three succesive $6$-neighbors\label{5-77666fig}}
\end{wrapfigure}
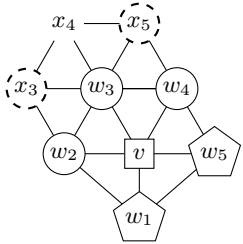

\textbf{Case 1: \bm{$x_4$} is a \bm{$5$}-vertex.}  Now $x_3$ and $x_5$ are both $7^+$-vertices, or else $G[N[w_3]]$ contains 
(RC\subref{5263a-fig}).  Furthermore, $w_2$ and $w_4$ both have $v$ as their unique $5$-neighbor, or else $G$ contains 
(RC\subref{5362a-fig}).  So now $x_3$ gives $1/6$ to $w_2$, and on to $v$; and $x_5$ gives $1/6$ to $w_4$, and on to $v$.  
Thus, $v$ gets $1/3+1/3+1/6+1/6=1$, and we are done.

\textbf{Case 2: \bm{$x_4$} is a \bm{$6$}-vertex.}  Either $x_3$ or $x_5$ is a $7^+$-vertex, or else $G[N[w_3]]$ contains 
(RC\subref{67a-fig}); 
by symmetry, assume that $x_3$ is a $7^+$-vertex.  So $x_3$ gives $1/6$ to $v$ via $w_3$.  If $x_5$ is also a $7^+$-vertex, 
then we are done; so assume that $x_5$ is a $6$-vertex.  If $w_2$ has no $5$-neighbor other than $v$, then $w_2$ also gets 
$1/6$ from $x_3$, and gives this all to $v$.  Thus, we assume $w_2$ has a $5$-neighbor other than $v$.  This implies 
that $w_4$ has $v$ as its unique $5$-neighbor, or $G$ contains (RC\subref{5363a-fig}).  
But now Lemma~\ref{helper1-lem} implies that $x_5$ gives $1/6$ to $v$ via $w_3$ and $w_4$; so $x$ gets $1/3+1/3+1/6+1/6$, 
and we are done.

\textbf{Case 3: \bm{$x_4$} is a \bm{$7^+$}-vertex.}  If $w_2$ and $w_4$ both have $5$-neighbors besides $v$, say $y'$ and $y''$, then
$G[\{v,w_2,w_3,w_4,y',y''\}]$ contains (RC\subref{5363a-fig}).  So, assume by symmetry that $w_2$ has no $5$-neighbor but $v$.  
Thus, by Lemma~\ref{helper1-lem}, vertex $x_3$ gives a total of $1/6$ to $v$ via $w_2$ and $w_3$.  Since $x_4$ also gives 
$1/6$ to $v$ via $w_3$, vertex $v$ gets $1/3+1/3+1/6+1/6=1$, and we are done.
\hfill$\qed$

\bigskip

The next lemma is our last, and will conclude our proof.

\begin{lem}
Each $5$-vertex with two non-adjacent $7$-neighbors ends happy.
\end{lem}
\noindent
\textit{Proof.}
Assume that $v$ is a $5$-vertex with neighbors $w_1,\ldots,w_5$ in order, and that $w_2$ and $w_5$ are $7$-vertices.
By Lemma~\ref{5-5lem}, we assume that $v$ has no $5$-neighbors.  If $v$ has another $7^+$-neighbor, then $v$ gets $1/3$ 
from each and we are done.  So we assume $w_1, w_3, w_4$ are $6$-vertices.  We label the remaining vertices as in 
Figure~\ref{5-76766afig}.  Note that $x_5$ is a $6^+$-vertex, or else $G$ contains (RC\subref{5262a-fig}).
Note also that $v$ receives $1/3$ from each of $w_2$ and $w_5$; so it suffices to show that $v$ receives at least an
additional $1/3$.  This is what we do next.

\begin{wrapfigure}[12]{l}{0.3\textwidth}
\centering
\begin{tikzpicture}[rotate=90, rotate=30, yscale=.866, xscale=-.5, scale=.985]
\tikzset{every node/.style=6vertB}

\draw 
(1,1) node[unvertB] (x3) {\footnotesize{$x_3$}}
--++ (1,-1) node (w3) {\footnotesize{$w_3$}}
--++ (1,1) node[7vertB] (w2) {\footnotesize{$w_2$}}
--++ (1,-1) node[5vertB] (v) {\footnotesize{$v$}} 
(x3) --++ (1,1) node[unvertB] (x2) {\footnotesize{$x_2$}} -- (w2)
(x3) --++ (-1,-1) node[unvertB] (x4) {\footnotesize{$x_4$}} -- (w3)
(w3) -- (v) (x3) -- (w2);
\draw (0,2) node[unvertB] (y3) {\footnotesize{$y_3$}};
\draw (x2) -- (y3) -- (x3);
\draw (3,-1) node (w4) {\footnotesize{$w_4$}};
\draw (5,-1) node[7vertB] (w5) {\footnotesize{$w_5$}};
\draw (5.5,0.5) node (w1) {\footnotesize{$w_1$}};
\draw (w3) -- (w4) -- (w5) -- (w1) 
(w2) -- (w1) -- (v) -- (w5) (w3) -- (w4) -- (v);
\draw (1,-1) node (x5) {\footnotesize{$x_5$}};
\draw (x4) -- (x5) -- (w4) (w3) -- (x5);
\draw (-1,1) node[unvertB] (y4) {\footnotesize{$y_4$}};
\draw (y3) -- (y4) -- (x4) (y4) -- (x3);
\draw (2,-2) node[unvertB] (x6) {\footnotesize{$x_6$}};
\draw (4,-2) node[unvertB] (x7) {\footnotesize{$x_7$}};
\draw (x5) -- (x6) -- (w4) -- (x7) -- (w5) (x6) -- (x7);
\draw (6,-2) node[unvertB] (x8) {\footnotesize{$x_8$}};
\draw (x7) -- (x8) -- (w5);
\end{tikzpicture}
\captionsetup{width=.275\textwidth}
\caption{A 5-vertex with two non-adjacent $7$-neighbors and three non-succesive $6$-neighbors\label{5-76766afig}}
\end{wrapfigure}

\textbf{Case 1: \bm{$x_5$} is a \bm{$6$}-vertex.}  
Now $x_4$ and $x_6$ are both also $6^+$-vertices, or else $G$ contains (RC\subref{5263a-fig}).  Suppose that $w_3$ and $w_4$ 
both have $v$
as their unique $5$-neighbor.  By Lemma~\ref{helper1-lem}, vertex $x_5$ gives $1/6$ to $v$ via $w_3$ and $w_4$.  If $x_3$ or $x_4$
is a $7^+$-vertex, then it gives $1/6$ to $v$ via $w_3$, and we are done; so we assume that $x_3$ and $x_4$ are both $6$-vertices.
Since $w_3$ has $v$ as its only $5$-neighbor, by Lemma~\ref{helper2-lem} vertex $x_4$ gives $1/6$ to $v$ via $w_3$.
In either case, $v$ receives at least $1/3+1/3+1/6+1/6=1$, so we are done.

So we assume instead that either $w_3$ or $w_4$ has a $5$-neighbor other than $v$.  They cannot both have second $5$-neighbors,
say $y'$ and $y''$, or else $G[\{v,w_3,w_4,y',y''\}]$ contains (RC\subref{5362a-fig}).  
So we assume by symmetry that $w_3$ has $v$ as its unique
$5$-neighbor.  Again, $x_4$ gives $1/6$ to $v$ via $w_3$ by Lemma~\ref{helper2-lem}. 
Since $w_4$ has a $5$-neighbor other than $v$, we know that $x_7$ is a $5$-vertex.  This implies that 
$x_6$ is a $7^+$-vertex; otherwise, $G$ contains (RC\subref{5264a-fig}).

If $x_3$ is a $7^+$-vertex, then it gives $1/6$ to $v$ via $w_3$; also $x_4$ gives $1/6$ to $v$ via $w_3$, 
by Lemma~\ref{helper2-lem}.  So we are done.  Thus, we assume that $x_3$ is a $6$-vertex.  Suppose that $x_4$ is a $7^+$-vertex.
So $x_5$ receives $1/6$ from each of $x_4$ and $x_6$.  Since $v$ receives $1/6$ from $x_4$ via $w_3$, it suffices to show that 
$x_5$ has at most one needy neighbor other than $w_3$ and $w_4$.  If both other neighbors of $x_5$ are needy, then $G$
contains (RC\subref{5365a-fig}), possibly with two $5_1$-vertices identified; this is forbidden.  Thus, $x_5$ splits its charge 
among at most $3$ needy neighbors, sending at least $(1/6+1/6)/3=1/9$ to each of $w_3$ and $w_4$.  All of this charge given to 
$w_3$ continues to $v$; and so does at least half the charge given to $w_4$.  Thus, in total $v$ gets $1/3+1/3+1/6+1/9+1/18=1$,
and we are done.

So assume instead that $x_4$ is a $6$-vertex; note that $x_4$ is not needy, or $G$ contains (RC\subref{5264a-fig}).  
Since $x_3, x_4, x_5$ are all $6$-vertices, $x_3$ has no $5$-neighbor, or else $G$ contains (RC\subref{5265a-fig}). 
Further, at most one of $x_2$ and $y_3$ is needy, or $G$ has a reducible configuration
(e.g., if $x_2$ and $y_3$ are $6$-vertices with $5$-neighbors $y'$ and $y''$, then $G[\{v,w_3,w_4,x_2,x_3,x_4,x_5,y_3,y',y''\}]$ 
contains (RC\subref{5367a-fig})).  
Also, $y_4$ is not needy or else $G$ contains (RC\subref{526571a-fig}), with the degree of the $7$-vertex decreased to $6$.
Thus, $x_3$ has at most two needy neighbors: $w_3$ and at most one of $x_2$ and $y_3$.  Furthermore, if $x_3$ has 
$w_2$ as its unique $7^+$-neighbor, then neither $x_2$ nor $y_3$ is needy, or else $G$ contains (RC\subref{5267c-fig}) or (RC\subref{5268b-fig}).  
Thus, $x_3$ sends $1/6$ to $v$ via $w_3$, and we are done.

\begin{wrapfigure}[12]{l}{0.3\textwidth}
\centering
\begin{tikzpicture}[rotate=90, rotate=30, yscale=.866, xscale=-.5]
\tikzset{every node/.style=6vertB}

\draw 
(1,1) node (x3) {\footnotesize{$x_3$}}
--++ (1,-1) node (w3) {\footnotesize{$w_3$}}
--++ (1,1) node[7vertB] (w2) {\footnotesize{$w_2$}}
--++ (1,-1) node[5vertB] (v) {\footnotesize{$v$}} 
(x3) --++ (1,1) node (x2) {\footnotesize{$x_2$}} -- (w2)
(x3) --++ (-1,-1) node (x4) {\footnotesize{$x_4$}} -- (w3)
(w3) -- (v) (x3) -- (w2);
\draw (3,-1) node (w4) {\footnotesize{$w_4$}};
\draw (5,-1) node[7vertB] (w5) {\footnotesize{$w_5$}};
\draw (5.5,0.5) node (w1) {\footnotesize{$w_1$}};
\draw (w3) -- (w4) -- (w5) -- (w1) 
(w2) -- (w1) -- (v) -- (w5) (w3) -- (w4) -- (v);
\draw[thick, dashed] (1,-1) node[7vertB] (x5) {\footnotesize{$x_5$}};
\draw (x4) -- (x5) -- (w4) (w3) -- (x5);
\draw (2,-2) node[unvertB] (x6) {\footnotesize{$x_6$}};
\draw (4,-2) node[unvertB] (x7) {\footnotesize{$x_7$}};
\draw (x5) -- (x6) -- (w4) -- (x7) -- (w5) (x6) -- (x7);
\draw (6,-2) node[unvertB] (x8) {\footnotesize{$x_8$}};
\draw (x7) -- (x8) -- (w5);
\end{tikzpicture}
\captionsetup{width=.275\textwidth}
\caption{A 5-vertex with two non-adjacent $7$-neighbors and three non-succesive $6$-neighbors\label{5-76766bfig}}
\end{wrapfigure}

\textbf{Case 2: \bm{$x_5$} is a \bm{$7^+$}-vertex.}  
See Figure~\ref{5-76766bfig}.  If both $w_3$ and $w_4$ have $v$ as their unique $5$-neighbor, then $x_5$
gives $1/6$ to $v$ via each of $w_3$ and $w_4$, and we are done.  So assume that $w_3$ or $w_4$ has another $5$-neighbor.  
If both have second $5$-neighbors, say $y'$ and $y''$, then $G[\{v,w_3,w_4,y',y''\}]$ contains (RC\subref{5362a-fig}).  
So assume that $w_3$ has $v$ as its unique $5$-neighbor, but $w_4$ has a $5$-neighbor other than $v$; call it $y'$. 
So $x_5$ sends $1/6$ to $v$ via $w_3$.  If $x_3$ or $x_4$ is a
$7^+$-vertex, then it gives $1/6$ to $v$ via $w_3$ and we are done.  So assume that $x_3$ and $x_4$ are both $6$-vertices.

If $x_3$ and $x_4$ are both needy, say with $5$-neighbors $y''$ and $y'''$, then $G[\{v,w_3,w_4,x_3,x_4,y',y'',y'''\}]$ contains (RC\subref{5464a-fig}).  So let $x'$ be a vertex among $x_3$ and $x_4$ that is not needy.
We first show that $x'$ gives $1/12$ to $v$ via $w_3$.  If $x'$ has at least two $7^+$-neighbors, then it receives at least $2(1/6)$,
and splits it at most four ways, giving $w_3$ at least $1/12$.  But if $x'$ has only a single $7^+$-neighbor, then none of its
neighbors but $w_3$ and $\{x_3,x_4\}\setminus\{x'\}$ is needy, or else $G$ contains (RC\subref{5365b-fig}) or (RC\subref{5366b-fig}) or
(RC\subref{5367b-fig}).
So $x'$ gets $1/6$ from its $7^+$-neighbor and sends at least half
to $v$ via $w_3$.  

If $x_5$ is an $8^+$-vertex, then it sends $1/4$ to $v$ via $w_3$.  So $v$ gets $1/3+1/3+1/4+1/12=1$.
So assume $x_5$ is a $7$-vertex.  Also $x_6$ is a $5$-vertex; otherwise $x_5$ sends $1/6$ to $w_4$ and $1/12$ to $v$.

Suppose that $x_4$ is not needy (and $x_3$ is possibly needy).
Now we show that $x_4$ sends $1/6$ to $v$ via $w_3$.  
If $x_4$ has at least three $7^+$-neighbors, then it receives $3(1/6)$ and sends $1/6$ to $w_3$.  
And if $x_4$ has $x_5$ as its only $7^+$-neighbor, then $G[N[\{x_4,x_5\}]]$ contains (RC\subref{516771a-fig}).
So instead assume that $x_4$ has exactly two $7^+$-neighbors: $x_5$ and exactly one of $y_4,y_5,y_6$.  
Now $w_3$ is needy, and possibly so is $x_3$; but no other neighbors
of $x_4$ are needy, or else $G$ contains (RC\subref{526571a-fig}) or (RC\subref{536471a-fig}) or (RC\subref{5366b-fig})
or (RC\subref{5365b-fig}).  To be specific, we label vertices as in Figure~\ref{5-76766cfig}.
We list the possible triples of a $7^+$-neighbor of $x_4$, a needy neighbor of $x_4$, and the reducible configuration:
$(y_4, y_5$, RC\subref{526571a-fig}),
$(y_4, y_6$, RC\subref{536471a-fig}),
$(y_5, y_4$, RC\subref{5365b-fig}),
$(y_5, y_6$, RC\subref{536471a-fig}),
$(y_6, y_4$, RC\subref{5365b-fig}),
$(y_6, y_5$, RC\subref{5366b-fig}).
Thus, $x_4$ sends $1/6$ to $v$ via $w_3$, as desired.
So $v$ gets $1/3+1/3+1/6+1/6=1$, and we are done.

Finally, assume that $x_3$ is not needy (and $x_4$ is needy); see Figure~\ref{5-76766cfig}.
If $y_6$ is a $5$-vertex, then $G[N[x_5]]$ contains (RC\subref{526371a-fig}).  
So assume $y_6$ is a a $6^+$-vertex.  Since $x_4$ is needy, $y_5$ is a $5$-vertex.  Thus, 

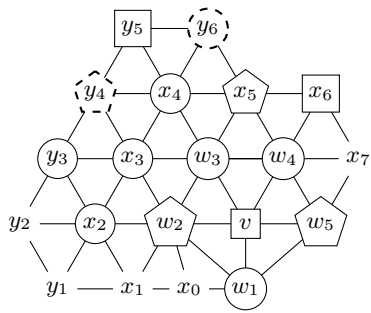
\begin{wrapfigure}[13]{l}{0.3\textwidth}
\centering
\begin{tikzpicture}[rotate=90, rotate=30, yscale=.866, xscale=-.5, scale=.985]
\tikzset{every node/.style=6vertB}

\draw 
(1,1) node (x3) {\footnotesize{$x_3$}}
--++ (1,-1) node (w3) {\footnotesize{$w_3$}}
--++ (1,1) node[7vertB] (w2) {\footnotesize{$w_2$}}
--++ (1,-1) node[5vertB] (v) {\footnotesize{$v$}}; 
\draw (x3) ++ (1,1) node (x2) {\footnotesize{$x_2$}};
\draw (x3) -- (x2) -- (w2)
(x3) --++ (-1,-1) node (x4) {\footnotesize{$x_4$}} -- (w3)
(w3) -- (v) (x3) -- (w2);
\draw (4,2) node[unvertB] (x1) {\footnotesize{$x_1$}} --
(3,3) node[unvertB] (y1) {\footnotesize{$y_1$}} --
(1,3) node[unvertB] (y2) {\footnotesize{$y_2$}}; 
\draw (0,2) node (y3) {\footnotesize{$y_3$}};
\draw (y2) -- (y3);
\draw (w2) -- (x1) -- (x2) -- (y1) (y2) -- (x2) -- (y3) -- (x3);
\draw (3,-1) node (w4) {\footnotesize{$w_4$}};
\draw (5,-1) node[7vertB] (w5) {\footnotesize{$w_5$}};
\draw (5.5,0.5) node (w1) {\footnotesize{$w_1$}};
\draw (w3) -- (w4) -- (w5) -- (w1) (w2) -- (w1) -- (v) -- (w5) (w3) -- (w4) -- (v);
\draw (barycentric cs:w1=1,x1=1) node[unvertB] (x0) {\footnotesize{$x_0$}};
\draw (x1) -- (x0) -- (w2) (x0) -- (w1);
\draw (1,-1) node[7vertB] (x5) {\footnotesize{$x_5$}};
\draw (x4) -- (x5) -- (w4) (w3) -- (x5);
\draw[thick, dashed] (-1,1) node[7vertB] (y4) {\footnotesize{$y_4$}};
\draw (y3) -- (y4) -- (x4) (y4) -- (x3);
\draw (2,-2) node[5vertB, inner sep=.25pt] (x6) {\footnotesize{$x_6$}};
\draw (4,-2) node[unvertB] (x7) {\footnotesize{$x_7$}};
\draw (x5) -- (x6) -- (w4) -- (x7) -- (w5) (x6) -- (x7);
\draw[thick, dashed] (x4) ++ (-1,-1) node (y6) {\footnotesize{$y_6$}};
\draw (y4) --++ (-1,-1) node[5vertB, inner sep=.5pt] (y5) {\footnotesize{$y_5$}} -- (x4) --++ (y6);
\draw (x5) -- (y6) -- (y5);
\end{tikzpicture}
\captionsetup{width=.275\textwidth}
\caption{A 5-vertex with two non-adjacent $7$-neighbors and three non-succesive $6$-neighbors\label{5-76766cfig}}
\end{wrapfigure}

\noindent
$y_4$ is a $7^+$-vertex, or else
$G[\{v,w_3,x_3,x_4,y_4,y_5\}]$ contains (RC\subref{5264a-fig}).
Now $x_3$ gets $1/6$ from each of $w_2$ and $y_4$.
If neither $x_2$ nor $y_3$ is needy, then $x_3$ sends $1/6$ to $v$ via $w_3$, and we are done.  
Similarly, if either $x_2$ or $y_3$ is a $7^+$-vertex, then $x_3$ again sends $1/6$ to $v$ via $w_3$.
So assume that both $x_2$ and $y_3$ are $6$-vertices and at least one of them is needy, with $5$-neighbor $z'$.  
Regardless, note that $x_3$ sends at least $(1/6+1/6)/4=1/12$ to $v$ via $w_3$.  Since $v$ also receives $1/3$ from 
each of $w_2$ and $w_5$, and also gets $1/6$ from $x_5$ via $w_3$, it suffices to show that $v$ receives an additional $1/12$.

If $w_1$ has a $5$-neighbor $z''$ other than $v$, 
then $G[\{v,w_1,w_2,w_3,x_2,$ $x_3,x_4,y_3,y_5,z',z''\}]$ contains (RC\subref{546571a-fig}) or (RC\subref{546671b-fig}); 
so assume not.  If $w_1$ has a $7^+$-neighbor besides $w_2$ and $w_5$, then it gets $1/6$ and sends it to $v$, 
and we are done.  So we assume $x_0$ is a $6$-vertex.  If $x_0$ is needy, then $G$ contains (RC\subref{536571a-fig}), 
since $x_4$ is needy.  So we assume $x_0$ is not needy.  If $x_1$ is a $6$-vertex,
then $G[N[w_2]]$ contains (RC\subref{6771a-fig}); so assume that $x_1$ is a $7^+$-vertex.  Now $x_0$ gets at least $2(1/6)$ 
and gives $v$ at least $1/12$ via $w_1$.  So $v$ gets $1/3+1/3+1/6+1/12+1/12=1$, and we are done.
\hfill$\qed$

\section*{Acknowledgments}
We thank two anonymous referees who read this manuscript quite closely and suggested a number of improvements.
In particular, one referee caught a number of inaccuracies in an earlier version.

\footnotesize{
\bibliographystyle{habbrv}
\bibliography{reconfiguration}
}

\end{document}